\documentclass[11pt,twoside,reqno]{amsart}

\usepackage{microtype}
\usepackage[OT1]{fontenc}
\usepackage{type1cm}
\usepackage{amssymb}
\usepackage{enumerate}
\usepackage{comment}
\usepackage{xcolor}
\usepackage{mathrsfs}
\usepackage{dsfont}
\usepackage{tikz}
\usepackage[normalem]{ulem}

\usetikzlibrary{arrows}

\usepackage[left=2.3cm,top=3cm,right=2.3cm]{geometry}
\numberwithin{equation}{section}

\theoremstyle{plain}
\newtheorem{theorem}{Theorem}[section]

\newtheorem{corollary}[theorem]{Corollary}
\newtheorem{proposition}[theorem]{Proposition}

\newtheorem{lemma}[theorem]{Lemma}

\theoremstyle{remark}

\newtheorem{example}[theorem]{Example}

\newtheorem*{acknowledgements}{Acknowledgement}

\theoremstyle{definition}

\newcommand{\BB}{\mathcal{B}}

\newcommand{\LL}{\mathcal{L}}

\newcommand{\CC}{\mathcal{C}}

\newcommand{\UU}{\mathcal{U}}

\renewcommand{\SS}{\mathcal{S}}
\newcommand{\R}{\mathbb{R}}
\newcommand{\RP}{\mathbb{RP}^1}

\newcommand{\N}{\mathbb{N}}
\newcommand{\hhh}{\mathtt{h}}
\newcommand{\iii}{\mathtt{i}}
\newcommand{\jjj}{\mathtt{j}}
\newcommand{\kkk}{\mathtt{k}}

\newcommand{\eps}{\varepsilon}

\newcommand{\A}{\mathsf{A}}
\newcommand{\B}{\mathsf{B}}

\newcommand{\dd}{\,\mathrm{d}}

\renewcommand{\ge}{\geqslant}
\renewcommand{\le}{\leqslant}
\renewcommand{\geq}{\geqslant}
\renewcommand{\leq}{\leqslant}

\DeclareMathOperator{\dist}{dist}
\DeclareMathOperator{\diam}{diam}

\DeclareMathOperator{\tr}{tr}

\DeclareMathOperator{\rank}{rank}

\DeclareSymbolFont{extraup}{U}{zavm}{m}{n}

\DeclareMathSymbol{\vardiamond}{\mathalpha}{extraup}{87}
\definecolor{blue}{rgb}{0,0,1}

\begin{document}

\title[Domination and thermodynamic formalism]{Domination, almost additivity, and thermodynamic formalism for planar matrix cocycles}

\author{Bal\'azs B\'ar\'any}
\address[Bal\'azs B\'ar\'any]
        {Budapest University of Technology and Economics \\
         MTA-BME Stochastics Research Group \\
         Department of Stochastics \\
         P.O.\ Box 91 \\
         1521 Budapest \\
         Hungary}
\email{balubsheep@gmail.com}

\author{Antti K\"aenm\"aki}
\address[Antti K\"aenm\"aki]
        {Department of Physics and Mathematics \\
         University of Eastern Finland \\
         P.O.\ Box 111 \\
         FI-80101 Joensuu \\
         Finland}
\email{antti.kaenmaki@uef.fi}

\author{Ian D. Morris}
\address[Ian D. Morris]
        {Department of Mathematics\\
	       University of Surrey\\
	       Guildford GU2 7XH\\
	       United Kingdom}
\email{i.morris@surrey.ac.uk}

\date{\today}

\begin{abstract}
In topics such as the thermodynamic formalism of linear cocycles, the dimension theory of self-affine sets, and the theory of random matrix products, it has often been found useful to assume positivity of the matrix entries in order to simplify or make feasible certain types of calculation. It is natural to ask how positivity may be relaxed or generalised in a way which enables similar calculations to be made in more general contexts. On the one hand one may generalise by considering \emph{almost additive} or \emph{asymptotically additive potentials} which mimic the properties enjoyed by the logarithm of the norm of a positive matrix cocycle; on the other hand one may consider matrix cocycles which are \emph{dominated}, a condition which includes positive matrix cocycles but is more general. In this article we explore the relationship between almost additivity and domination for planar cocycles. We show in particular that a locally constant linear cocycle in the plane is almost additive if and only if it is either conjugate to a cocycle of isometries, or satisfies a property slightly weaker than domination which is introduced in this paper. Applications to matrix thermodynamic formalism are presented.
\end{abstract}

\maketitle

\section{Introduction}

For the purposes of this article a \emph{linear cocycle} over a dynamical system $T \colon X \to X$ will be a skew-product
\[
  F \colon X \times \R^d \to X \times \R^d, \quad (x,p) \mapsto (Tx,\mathsf{A}(x) p),
\]
where $\mathsf{A} \colon X \to GL_d(\R)$ is continuous and $X$ is a compact metric space. Writing $\mathsf{A}^n_T(x) = \mathsf{A}(T^{n-1}x) \cdots \mathsf{A}(x)$, we thus have $F^n(x,p) = (T^nx, \mathsf{A}^n_T(x) p)$ for all $n \in \N$ and
\begin{equation} \label{eq:cocycle-id}
  \mathsf{A}^{m+n}_T(x) = \mathsf{A}^m_T(T^nx) \mathsf{A}^n_T(x)
\end{equation}
for all $m,n \in \N$. In numerous contexts it has been found useful to consider cocycles in which all of the matrices $\mathsf{A}(x)$ are \emph{positive}: we note for example such diverse articles as \cite{FurstenbergKesten1960,HueterLalley95,Jungers12,Pollicott10}. Under this assumption the cocycle satisfies the inequality
\[
  \left|\log \|\mathsf{A}^{m+n}_T(x)\| -\log \|\mathsf{A}^m_T(T^nx)\|-\log \| \mathsf{A}^n_T(x)\|\right| \leq C
\]
for some constant $C>0$ depending only on $\mathsf{A}$. This has led some authors to extend results for positive linear cocycles by considering, instead of a linear cocycle, a sequence of continuous functions $f_n \colon X \to \mathbb{R}$ satisfying the inequality
\[
  \left|f_{n+m}(x)-f_m(T^nx)-f_n(x)\right| \leq C
\]
for all $x \in X$ and $n,m \geq 1$. Such sequences of functions are referred to in the literature as \emph{almost additive} and have been investigated in \cite{Barreira06,BarreiraDoutor09,BomfimVarandas15,IommiYayama12,Yayama16}. The condition of almost additivity implies trivially a further property, \emph{asymptotic additivity} (see for example Feng and Huang \cite[Proposition~A.5]{FengHuang10}), which has been applied in \cite{Cao13,FengHuang10,IommiYayama17}.  In another category of works, positivity is replaced by the more general hypothesis of \emph{domination}: under this hypothesis there exists a continuous splitting $\mathbb{R}^d=\mathcal{U}(x)\oplus \mathcal{V}(x)$, which is preserved by the cocycle, such that $\|\mathsf{A}^n_T(x)u\| \geq Ce^{n\varepsilon} \|\mathsf{A}^n_T(x)v\|$ for all unit vectors $u \in \mathcal{U}(x)$ and $v\in \mathcal{V}(x)$, for some constants $C,\varepsilon>0$ (see \cite{BochiGourmelon09} and references therein). For linear cocycles the hypothesis of domination implies the hypothesis of almost additivity, but the converse is false, as can be seen trivially for the case of cocycles where all of the linear maps are isometries, or where all are equal to the identity. The purpose of this article is to explore precisely the relationship between domination and almost additivity in the context of locally constant two-dimensional linear cocycles over the shift. In this project we are motivated principally by applications to the topics of matrix thermodynamic formalism and the geometry of self-affine fractals.

We consider cocycles in the simplest non-commutative setting, namely in the case of planar matrices. A cocycle is dominated if and only if there is a uniform exponential gap between singular values of its iterates. This is equivalent to the existence of a strongly invariant multicone in the projective space; see \cite{AvilaBochiYoccoz2010,BochiGourmelon09}. Domination originates from \cite{Mane1978, Mane1984} and it is an important concept in differentiable dynamical systems; see \cite{BochiViana2005, BonattiDiazPujals2003}. Our contribution in this article to this line of research is to show that a planar matrix cocycle is dominated if and only if matrices are proximal and the norms in the generated sub-semigroup satisfy a certain multiplicativity property; see Corollary \ref{thm:dom-cor1}. Higher dimensions are more difficult: \cite[\S 4]{BochiGourmelon09} show that the connected components of the multicone need not be convex.

Of the several motivations for studying almost additive potentials, this article is concerned principally with thermodynamic formalism. In Theorem \ref{thm:holder} we will show that almost additive potentials arising from the norm potential of a two-dimensional locally-constant linear cocycle over the full shift can in almost all cases be studied simply by using the classical thermodynamic formalism. In fact, in our results, we are able to characterise all the properties of equilibrium states for these norm potentials by means of the properties of matrices. Theorem \ref{thm:fortuples} gives a positive answer to \cite[Question 7.4]{BaranyKaenmakiKoivusalo2017} in the two dimensional case. Furthermore, in Example \ref{example}, answering a folklore question, we show the existence of a quasi-Bernoulli equilibrium state which is not a Gibbs measure for any H\"older continuous potential.

\section{Preliminaries and statements of results}

For the remainder of this article we specialise to cocycles whose values are invertible two-dimensional real matrices. We take $\mathsf{A} \subset GL_2(\R)$, set $X=\mathsf{A}^\N$, denote the left shift on $X$ by $T$, and let $\mathsf{A}(x)$ be the first matrix in the infinite sequence $x \in X$. Let
\[
  F \colon X \times \R^d \to X \times \R^d, \quad (x,p) \mapsto (Tx,\mathsf{A}(x) p)
\]
be a linear cocycle over $T$. We see that $\mathsf{A}_T^n(x)$ is the product of $n$ first matrices in $x \in X$, and the cocycle identity \eqref{eq:cocycle-id} clearly holds. Let $\mathcal{S}(\A)$ denote the sub-semigroup generated by $\A$, that is, $\mathcal{S}(\A)=\{A_1\cdots A_n:n\in\N\text{ and }A_i\in\A\text{ for all }i\in\{1,\ldots,n\}\}$. So in particular, $\mathsf{A}_T^n(x) \in\mathcal{S}(\mathsf{A})$ for all $x = (A_1,A_2,\ldots) \in X$ and $n\in\N$.

\subsection{Domination} \label{sec:subdom}
Following \cite{BochiGourmelon09} we say that a compact and nonempty subset $\mathsf{A}\subset GL_2(\mathbb{R})$ is \emph{dominated} if there exist constants $C>0$ and $0<\tau<1$ such that
$$
\frac{|\det(A_1\cdots A_n)|}{\|A_1\cdots A_n\|^2}\leq C \tau^{n}
$$
for all $A_1,\dots,A_n\in\A$. We let $\RP$ denote the real projective line, which is the set of all lines through the origin in $\R^2$. We call a proper subset $\CC\subset\RP$ a \emph{multicone} if it is a finite union of closed projective intervals. We say that $\CC\subset\RP$ is a \emph{strongly invariant multicone} for $\A\subset GL_2(\R)$ if it is a multicone and $A\CC\subset\CC^o$ for all $A\in\A$. Here $\CC^o$ is the interior of $\CC$. By \cite[Theorem~B]{BochiGourmelon09}, a compact set $\A \subset GL_2(\R)$ has a strongly invariant multicone if and only if $\A$ is dominated. We say that $\CC\subset\RP$ is an \emph{invariant multicone} for $\A\subset GL_2(\R)$ if it is a multicone and $A\CC\subset\CC$ for all $A \in \A$.

Recall that a matrix $A$ is \emph{proximal} if it has two real eigenvalues with unequal absolute values, \emph{parabolic} if it has only one eigenspace, i.e.\ the single eigenvalue has geometric multiplicity one, and \emph{conformal} if it has two eigenvalues with the same absolute values. In other words, a matrix $A$ is conformal if and only if there exists an invertible matrix $M$, which we call a \emph{conjugation matrix} of $A$, such that $|\det(A)|^{-1/2}MAM^{-1}\in O(2)$, where $O(2)$ is the group of $2 \times 2$ orthogonal matrices. Furthermore, we say that a set $\A\subset GL_2(\R)$ is \emph{strongly conformal} if all the elements of $\A$ are conformal with respect to the same conjugation matrix. Strongly conformality is equivalent to the fact that all the elements in the generated semigroup are conformal.


For a proximal  matrix $A$, let $\lambda_u(A)$ and $\lambda_s(A)$ be the largest and smallest eigenvalues of $A$ in absolute value, respectively. If the eigenvalues are equal in absolute value, then the choice of $\lambda_u(A)$ and $\lambda_s(A)$ is arbitrary. Note that if $A$ is diagonalisable, then there exist linearly independent subspaces $u(A),s(A)\in\RP$ such that $|\lambda_u(A)|=\|A|u(A)\|$ and $|\lambda_s(A)|=\|A|s(A)\|$. We call $u(A)\in\RP$ the eigenspace of $A$ corresponding to $\lambda_u(A)$ and $s(A)\in\RP$ the eigenspace corresponding to $\lambda_s(A)$. If $\A\subset GL_2(\R)$, then we define $X_u(\A)$ and $X_s(\A)$ to be the closures of the sets of all unstable and stable directions of proximal elements of $\mathcal{S}(\A)$, i.e.\ the sets
\begin{align*}
  X_u(\A)&=\overline{\{u(A): A\in\SS(\A)\text{ is proximal}\}}, \\
  X_s(\A)&=\overline{\{s(A):A\in\SS(\A)\text{ is proximal}\}},
\end{align*}
respectively. Recall that $\mathcal{S}(\mathsf{A})$ is the sub-semigroup of $GL_2(\mathbb{R})$ generated by $\mathsf{A}$, i.e. the set of all finite products formed by the elements of $\A$. We say that $\A\subset GL_2(\R)$ has an \emph{unstable multicone $\CC$} if $\SS(\A)$ contains at least one proximal element and
\begin{enumerate}
  \item\label{i-X3} $\CC\cap X_s(\A)=\emptyset$,
  \item\label{i-X4} $\partial\CC\cap X_u(\A)=\emptyset$,
  \item\label{i-X5} each connected component of $\CC$ intersects $X_u(\A)$.
\end{enumerate}
Finally, we say that a semigroup $\mathcal{S}\subset GL_2(\R)$ is \emph{almost multiplicative} if there exists a constant $\kappa>0$ such that $\|AB\|\geq \kappa \|A\|\|B\|$ for all $A,B \in \mathcal{S}$. We note that since clearly $\|AB\| \leq \|A\|\|B\|$ for all $A,B \in \mathcal{S}(\mathsf{A})$ for every $\mathsf{A} \subset GL_2(\mathbb{R})$, the condition $\|AB\|\geq \kappa \|A\|\|B\|$ for all $A,B \in \mathcal{S}(\mathsf{A})$ is equivalent to the statement that every cocycle taking values in $\mathcal{S}(\mathsf{A})$ is almost additive in the sense defined in the introduction.

Our main result for matrix cocycles is the following theorem.

\begin{theorem}\label{thm:justdomin}
  Let $\mathsf{A}\subset GL_2(\mathbb{R})$. If the sub-semigroup $\mathcal{S}(\A)$ is almost multiplicative, then exactly one of the two following conditions hold:
  \begin{enumerate}
    \item $\A$ is strongly conformal,
    \item\label{thm:item2} $\A$ has an invariant unstable multicone and $\SS(\A)$ does not contain parabolic elements.
  \end{enumerate}
\end{theorem}

The next two propositions show that if the proximal elements of $\A$ form a compact set, then the converse claim holds in Theorem \ref{thm:justdomin}.

\begin{proposition}\label{prop:connect}
  Let $\A\subset GL_2(\R)$ be such that $\A$ has an invariant unstable multicone and $\SS(\A)$ does not contain parabolic elements. Let $\A_e$ be the collection of all conformal elements of $\A$. Then
  \begin{enumerate}
    \item\label{item:that} $\A\setminus\A_e$ is nonempty and contains only proximal elements,
    \item\label{item:this} $\A_e$ is strongly conformal and $\mathcal{S}(\{|\det(A)|^{-1/2}A : A \in \mathsf{A}_e\})$ is finite.
  \end{enumerate}
    Moreover, if $\A\setminus\A_e$ is compact, then $\A\setminus\A_e$ has a strongly invariant multicone $\CC$ such that $A\CC=\CC$ for all $A\in\A_e$.
\end{proposition}

\begin{proposition}\label{prop:converse}
  Let $\mathsf{A}_e, \mathsf{A}_h \subset GL_2(\mathbb{R})$ be such that
  \begin{enumerate}
    \item $\mathsf{A}_h$ is nonempty, compact, and has a strongly invariant multicone $\CC$,
    \item $\A_e$ is strongly conformal and $A\CC=\CC$ for all $A\in\A_e$.
  \end{enumerate}
  Then $\mathcal{S}(\mathsf{A}_e\cup\mathsf{A}_h)$ is almost multiplicative.
\end{proposition}

The previous three statements have two immediate corollaries. The first one studies the case where $\A$ contains only proximal elements. The second one is for finite collections.

\begin{corollary}\label{thm:dom-cor1}
  If $\mathsf{A}\subset GL_2(\mathbb{R})$ is compact, then the following two statements are equivalent:
  \begin{enumerate}
    \item $\mathsf{A}$ has a strongly invariant multicone,
    \item $\mathsf{A}$ contains only proximal elements and $\mathcal{S}(\mathsf{A})$ is almost multiplicative.
  \end{enumerate}
\end{corollary}

\begin{corollary}\label{thm:justdomin2}
  If $\mathsf{A}\subset GL_2(\mathbb{R})$ is finite, then the following two statements are equivalent:
  \begin{enumerate}
  \item\label{it:domin1} the sub-semigroup $\mathcal{S}(\mathsf{A})$ is almost multiplicative,
  \item $\A$ can be decomposed into two sets $\A_e$ and $\A_h$ such that $\A_e$ is strongly conformal and if $\A_h\neq\emptyset$, then $\mathsf{A}_h$ has a strongly invariant multicone $\CC$ such that $A\CC=\CC$ for all $A\in\A_e$.
  \end{enumerate}
\end{corollary}

\subsection{Thermodynamic formalism}

If the set $\mathsf{A} \subset GL_2(\R)$ is finite, then it makes sense to consider thermodynamic formalism for matrix cocycles. In this context, it is rather standard practise to use separate alphabet to index the elements in the sub-semigroup.

Let $N \ge 2$ be an integer and $\Sigma = \{ 1,\ldots,N \}^\N$ be the collection of all infinite words obtained from integers $\{ 1,\ldots,N \}$. We denote the left shift operator by $\sigma$ and equip $\Sigma$ with the product discrete topology. The \emph{shift space} $\Sigma$ is clearly compact. If $\iii = i_1i_2\cdots \in \Sigma$, then we define $\iii|_n = i_1 \cdots i_n$ for all $n \in \N$. The empty word $\iii|_0$ is denoted by $\varnothing$. Define $\Sigma_n = \{ \iii|_n : \iii \in \Sigma \}$ for all $n \in \N$ and $\Sigma_* = \bigcup_{n \in \N} \Sigma_n \cup \{ \varnothing \}$. Thus $\Sigma_*$ is the collection of all finite words. The length of $\iii \in \Sigma_* \cup \Sigma$ is denoted by $|\iii|$. If $\iii \in \Sigma_n$ for some $n$, then we set $[\iii] = \{ \jjj \in \Sigma : \jjj|_n = \iii \}$. The set $[\iii]$ is called a \emph{cylinder set}. Cylinder sets are open and closed and they generate the Borel $\sigma$-algebra.

The longest common prefix of $\iii,\jjj \in \Sigma_*
\cup \Sigma$ is denoted by $\iii \wedge \jjj$. The concatenation of two words $\iii \in \Sigma_*$ and $\jjj \in \Sigma_* \cup \Sigma$ is denoted by $\iii\jjj$. If $A \subset \Sigma$ and $\iii \in \Sigma_*$, then $\iii A = \{ \iii\jjj : \jjj \in A \}$. For example, if $\iii,\jjj \in \Sigma_*$, then $[\iii\jjj] = \iii[\jjj] = \iii\jjj\Sigma$. If $\iii \in \Sigma_*$ and $n \in \N$, then by $\iii^n$ we mean the concatenation $\iii\cdots\iii$ where $\iii$ is repeated $n$ times. Finally, denote by $\sharp_k\iii$ the number of appearances of the symbol $k \in \{ 1,\ldots,N \}$ in $\iii \in \Sigma_*$, i.e.\ $\sharp_k\iii=\sharp\{n:i_n=k\text{ for }n \in \{1,\ldots|\iii|\}\}$.

We say that the sequence $\Phi = (\phi_n)_{n \in \N}$ of functions $\phi_n \colon \Sigma \to \R$ is \emph{sub-additive} if there exists $C_1 \ge 0$ such that
\begin{equation*}
  \phi_{n+m}(\iii) \le \phi_n(\iii) + \phi_m(\sigma^n\iii) + C_1
\end{equation*}
for all $n,m \in \N$ and $\iii \in \Sigma$. A sub-additive sequence $\Phi = (\phi_n)_{n \in \N}$ is \emph{almost-additive} if there exists $C_2 \ge 0$ such that
\begin{equation*}
  \phi_{n+m}(\iii) \ge \phi_n(\iii) + \phi_m(\sigma^n\iii) - C_2
\end{equation*}
for all $n,m \in \N$ and $\iii \in \Sigma$. Finally, we say that an almost-additive sequence $\Phi$ is \emph{additive} if the constants $C_1$ and $C_2$ in the above inequalities can be chosen to $0$. For example, if $\phi \colon \Sigma \to \R$ is a function, then $(\sum_{k=0}^{n-1} \phi \circ \sigma^k)_{n \in \N}$ is additive. In this context, the function $\phi$ is called a \emph{potential}. We say that a potential $\phi$ is \emph{H\"older continuous}, if there exist $C>0$ and $0<\tau<1$ such that
$$
  |\phi(\iii)-\phi(\jjj)|\leq C\tau^{|\iii\wedge\jjj|}.
$$
for all $\iii,\jjj\in\Sigma$.

If $\Phi = (\phi_n)_{n \in \N}$ is sub-additive, then the \emph{pressure} of $\Phi$ is defined by
\begin{equation}\label{eq:pressure}
  P(\Phi) = \lim_{n \to \infty} \tfrac{1}{n} \log \sum_{\iii \in \Sigma_n} \exp\max_{\jjj \in [\iii]}\phi_n(\jjj).
\end{equation}
The limit above exists by the standard properties of sub-additive sequences. Let $\mu$ be a $\sigma$-invariant probability measure on $\Sigma$ and recall that the \emph{Kolmogorov-Sinai entropy} of $\mu$ is
\begin{equation*}
  h_\mu = -\lim_{n \to \infty} \tfrac{1}{n} \sum_{\iii \in \Sigma_n} \mu([\iii]) \log\mu([\iii]).
\end{equation*}
In addition, if $\Phi = (\phi_n)_{n \in \N}$ is a sub-additive sequence, then we set
\begin{equation*}
  \Lambda_\mu(\Phi) = \lim_{n \to \infty} \tfrac{1}{n} \int_\Sigma \phi_n(\iii) \,\mathrm{d}\mu(\iii).
\end{equation*}
It is easy to see that
\begin{equation*}
  P(\Phi) \ge h_\mu + \Lambda_\mu(\Phi)
\end{equation*}
for all $\sigma$-invariant probability measures $\mu$. The variational principle
\begin{equation*}
  P(\Phi) =\sup\left\{ h_\mu + \Lambda_\mu(\Phi):\mu\text{ is $\sigma$-invariant and }\Lambda_\mu(\Phi)\neq-\infty\right\}
\end{equation*}
is proved in \cite{CaoFengHuang}. For matrix cocycles this was obtained earlier in \cite{Kaenmaki2004}. A $\sigma$-invariant measure $\mu$ satisfying
\begin{equation*}
  P(\Phi) = h_\mu + \Lambda_\mu(\Phi)
\end{equation*}
is called an \emph{equilibrium state} for $\Phi$. Such a measure always exists in the context of matrix cocycles, but it is not known if a general sub-additive sequence has an equilibrium state; see \cite{Barreira2010}.

We say that a probability measure $\mu$ on $\Sigma$ is \emph{quasi-Bernoulli} if there exists a constant $C\ge 1$ such that
$$
  C^{-1}\mu([\iii])\mu([\jjj])\leq\mu([\iii\jjj])\leq C\mu([\iii])\mu([\jjj])
$$
for all $\iii,\jjj\in\Sigma_*$. If the constant $C$ above can be chosen to $1$, then $\mu$ is a \emph{Bernoulli} measure. In other words, a probability measure $\mu$ is Bernoulli if there exist a probability vector $(p_1,\ldots,p_N)$ such that
$$
  \mu([\iii])=p_{i_1}\cdots p_{i_{n}}
$$
for all $\iii=i_1\cdots i_n\in\Sigma_n$ and $n \in \N$.

Let $\phi \colon \Sigma \to \R$ be a continuous potential and $\Phi = (\sum_{k=0}^{n-1} \phi \circ \sigma^k)_{n \in \N}$. We say that a Borel probability measure $\mu$ on $\Sigma$ is a \emph{Gibbs measure} for $\phi$ if there exists a constant $C \ge 1$ such that
\begin{equation}\label{eq:gibbsmeas}
  C^{-1}\exp\biggl(-nP(\Phi) + \sum_{k=0}^{n-1} \phi(\sigma^k(\jjj))\biggr) \le \mu([\iii]) \le C\exp\biggl(-nP(\Phi) + \sum_{k=0}^{n-1} \phi(\sigma^k(\jjj))\biggr)
\end{equation}
for all $\iii \in \Sigma_n$, $\jjj \in [\iii]$, and $n \in \N$. For example, the Bernoulli measure obtained from a probability vector $(p_1,\ldots,p_N)$ is a Gibbs measure for the potential $\iii\mapsto\log p_{\iii|_1}$. If $\phi$ is H\"older continuous, then there is unique $\sigma$-invariant Gibbs measure which also is unique equilibrium state; see \cite[Theorems 1.4 and 1.22]{Bowen}.

Similarly, if $\Phi = (\phi_n)_{n \in \N}$ is sub-additive, then a Borel probability measure $\mu$ on $\Sigma$ is a \emph{Gibbs-type measure} for $\Phi$ if there exists a constant $C \ge 1$ such that
\begin{equation}\label{eq:Gibbstype}
  C^{-1}\exp\biggl(-nP(\Phi) + \phi_n(\jjj)\biggr) \le \mu([\iii]) \le C\exp\biggl(-nP(\Phi) + \phi_n(\jjj)\biggr)
\end{equation}
for all $\iii \in \Sigma_n$, $\jjj \in [\iii]$, and $n \in \N$. It is easy to see that a $\sigma$-invariant Gibbs-type measure is ergodic and hence the unique equilibrium state; see \cite[\S 3.2]{KaenmakiReeve2014}. If $\Phi$ is almost-additive, then, similarly as with continuous potentials, there exist conditions to guarantee the existence of a $\sigma$-invariant Gibbs-type; see \cite[\S 4.2]{Barreira2010}.

Our main objective is to study thermodynamic formalism in the setting of matrix cocycles. Let $\mathsf{A} = (A_1,\ldots,A_N) \in GL_2(\mathbb{R})^N$, $s > 0$, and define $\phi_n^s \colon \Sigma \to \R$ for all $n \in \N$ by setting $\phi_n^s(\iii) = \log\| A_{\iii|_n} \|^s$, where $A_\iii = A_{i_1} \cdots A_{i_n}$ for all $\iii = i_1 \cdots i_n \in \Sigma_n$ and $n \in \N$. Then the sequence $\Phi^s = (\phi_n^s)_{n \in \N}$ parametrised by $s > 0$ is sub-additive. By \cite[Theorems~2.6 and 4.1]{Kaenmaki2004}, for every choice of the matrix tuple $\A$, there exists an ergodic equilibrium state for $\Phi^s$. The structure of the set of all equilibrium states for $\Phi^s$ is well known. We say that $\A = (A_1,\ldots,A_N) \in GL_2(\R)^N$ is \emph{irreducible} if there does not exist $1$-dimensional linear subspace $V$ such that $A_iV=V$ for all $i \in \{ 1,\ldots,N \}$; otherwise $\A$ is \emph{reducible}. In a reducible tuple $\A$, all the matrices are simultaneously upper triangular in some basis. If $\A$ is irreducible, then there is unique equilibrium state which is a Gibbs-type measure for $\Phi^s$; see \cite[Proposition 1.2]{FengKaenmaki2011}. It is worthwhile to remark that irreducibility does not imply that $\Phi^s$ is almost-additive. In the reducible case, there can be two distinct ergodic equilibrium states; see \cite[Theorem 1.7]{FengKaenmaki2011}. Recall also that the set $\{ \A \in GL_2(\R)^N : \A \text{ is irreducible} \}$ is open, dense, and of full Lebesgue measure in $GL_2(\R)^N$. In fact, the complement of the set is a finite union of $(4N-1)$-dimensional algebraic varieties; see \cite[Propositions 3.4 and 3.6]{KaenmakiLi2017}.


The following four results characterise different kind of properties equilibrium states for $\Phi^s$ can have by means of the matrix tuple.

\begin{proposition}\label{prop:gibbstype}
  If $\A=(A_1,\ldots,A_N)\in GL_2(\R)^N$ and $\mu$ is an ergodic equilibrium state for $\Phi^s$, then the following two statements are equivalent:
  \begin{enumerate}
    \item\label{it:gibbstype1} $\mu$ is a Gibbs-type measure for $\Phi^s$,
    \item\label{it:gibbstype2} at least one of the following three conditions hold:
    \begin{enumerate}
    	\item\label{casea} $\A$ is irreducible,
    	\item\label{caseb} $\A$ is strongly conformal,
    	\item\label{casec} $\A$ is reducible with a common invariant subspace $V$ and there exists $\varepsilon>0$ such that either the closed $\varepsilon$-neighbourhood of $V$ or the closure of its complement is an invariant unstable multicone.
    \end{enumerate}
  \end{enumerate}
\end{proposition}

Note that $\A$ can be both irreducible and strongly conformal and that neither condition imply each other.

\begin{proposition}\label{prop:bernoulli}
  If $\A=(A_1,\ldots,A_N)\in GL_2(\R)^N$ and $\mu$ is an ergodic equilibrium state for $\Phi^s$, then the following two statements are equivalent:
  \begin{enumerate}
  \item\label{it:bernoulli1} $\mu$ is a Bernoulli measure,
  \item\label{it:bernoulli2} $\A$ is reducible or $\mathsf{A}$ is strongly conformal.
  \end{enumerate}
\end{proposition}

In the previous two propositions, one has to assume that the equilibrium measure is ergodic; see \cite[Example 6.2]{KaenmakiVilppolainen2010} for a counter-example. We remark that the Bernoulli property has been studied earlier in \cite[Theorem 13]{Morris2016}.
Since the propositions give a complete characterisation of the properties in the reducible case, we can restrict our attention into irreducible matrix tuples.

\begin{theorem}\label{thm:fortuples}
  If $\A=(A_1,\ldots,A_N)\in GL_2(\R)^N$ is irreducible and $\mu$ is an equilibrium state for $\Phi^s$, then the following four statements are equivalent:
  \begin{enumerate}
    \item\label{thm:ftv} $\mu$ is a quasi-Bernoulli measure,
    \item\label{thm:ftvvv} $\mathcal{S}(\A)$ is almost multiplicative,
    \item\label{thm:fti} $\A$ can be decomposed into two sets $\A_e$ and $\A_h$ such that $\A_e$ is strongly conformal and if $\A_h\neq\emptyset$, then $\mathsf{A}_h$ has a strongly invariant multicone $\CC$ such that $A\CC=\CC$ for all $A\in\A_e$,
    \item\label{thm:ftvii} there exist a constant $C>0$ and a $\mu$-almost everywhere continuous potential $f\in L^1(\mu)$ such that
    \begin{equation}\label{eq:shadowing}
      \Biggl|\sum_{k=0}^{n-1}f(\sigma^k\iii)-\log\|A_{\iii|_n}\|\Biggr|\leq C
    \end{equation}
    for all $\iii \in \Sigma$ and $n \in \N$.
  \end{enumerate}
\end{theorem}

The previous theorem gives a positive answer to \cite[Question 7.4]{BaranyKaenmakiKoivusalo2017} in the two dimensional case.

\begin{theorem}\label{thm:holder}
  If $\A=(A_1,\ldots,A_N)\in GL_2(\R)^N$ is irreducible and $\mu$ is an equilibrium state for $\Phi^s$, then the following three statements are equivalent:
  \begin{enumerate}
    \item\label{it:holder1} $\mu$ is a Gibbs measure for some H\"older continuous potential,
    \item\label{it:holder2} $\A$ has a strongly invariant multicone or $\mathsf{A}$ is strongly conformal,
    \item\label{it:holder3} there exist a constant $C>0$ and a H\"older-continuous potential $f$ such that
    $$
      \Biggl|\sum_{k=0}^{n-1}f(\sigma^k\iii)-\log\|A_{\iii|_n}\|\Biggr|\leq C
    $$
    for all $\iii \in \Sigma$ and $n \in \N$.
  \end{enumerate}
\end{theorem}

\begin{center}
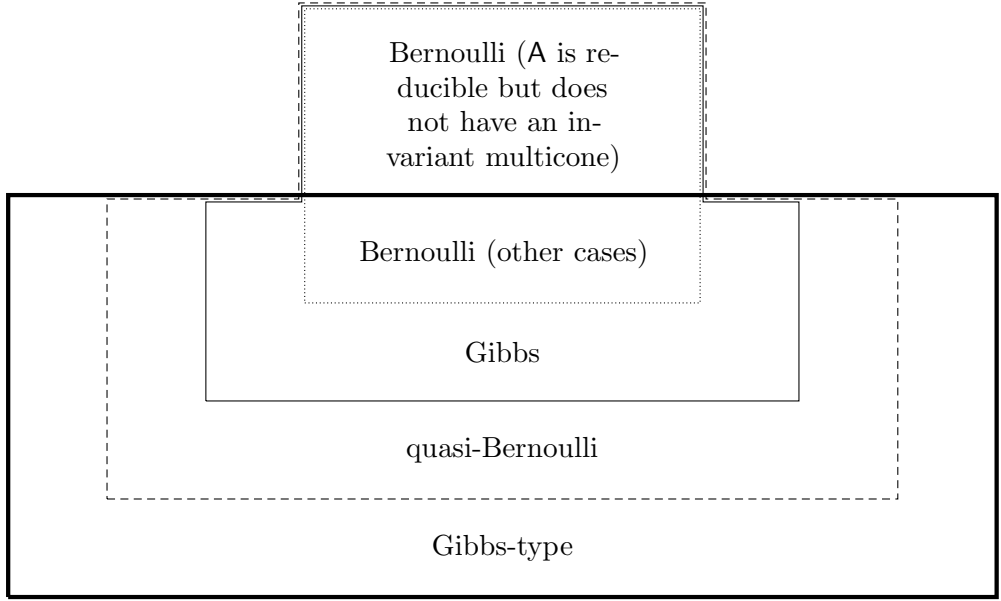
\begin{figure}[t]
\begin{tikzpicture}[scale=1.3]
    \draw [densely dotted] (3,3) rectangle (7,6);
    \node[align=center, text width=4cm] at (5,5) {Bernoulli ($\mathsf{A}$ is reducible but does not have an invariant multicone)};
    \node[align=center, text width=4cm] at (5,3.5) {Bernoulli (other cases)};

    \draw (2,2) -- (8,2) -- (8,4.03) -- (7.03,4.03) -- (7.03,6.03) -- (2.97,6.03) -- (2.97,4.03) -- (2,4.03) -- (2,2);
    \node[align=center] at (5,2.5) {Gibbs};

    \draw[densely dashed] (1,1) -- (9,1) -- (9,4.06) -- (7.06,4.06) -- (7.06,6.06) -- (2.94,6.06) -- (2.94,4.06) -- (1,4.06) -- (1,1);
    \node[align=center] at (5,1.5) {quasi-Bernoulli};

    \draw[ultra thick] (0,0) -- (10,0) -- (10,4.10) -- (0,4.10) -- (0,0);
    \node[align=center] at (5,0.5) {Gibbs-type};

  \end{tikzpicture}
  \caption{Classification of equilibrium states for $\Phi^s$.}
  \label{fig:illustration}
\end{figure}
\end{center}

Figure \ref{fig:illustration} illustrates how different properties of equilibrium states for $\Phi^s$ are related. The following example shows that the inclusions depicted in the figure are strict.

\begin{example} \label{example}
  (1) It can happen that an equilibrium state for $\Phi^s$ is a Gibbs measure for some H\"older-continuous potential, but is not a Bernoulli measure: Choose two positive matrices
  \begin{equation*}
    A_1 =
    \begin{pmatrix}
      2 & 1 \\
      1 & 1
    \end{pmatrix}
    \quad\text{and}\quad
    A_2 =
    \begin{pmatrix}
      2 & 1 \\
      1 & 2
    \end{pmatrix}.
  \end{equation*}
  Then $(A_1,A_2)$ is irreducible and has a strongly invariant multicone (i.e.\ the union of the first and third quadrants). The claim follows now from Theorem \ref{thm:holder} and Proposition \ref{prop:bernoulli}.

  (2) It can happen that an equilibrium state for $\Phi^s$ is a quasi-Bernoulli measure, but is not a Gibbs measure for any H\"older-continuous potential: Let $A_1$ and $A_2$ be as above. Then $(A_1,A_2,I)$ is irreducible and has an invariant multicone (i.e.\ the union of the first and third quadrants). The claim follows now from Theorems \ref{thm:fortuples} and \ref{thm:holder}.

  (3) It can happen that an equilibrium state for $\Phi^s$ is a Gibbs-type measure for $\Phi^s$, but is not a quasi-Bernoulli measure: Choose two matrices
  \begin{equation*}
    A_3 =
    \begin{pmatrix}
      1 & 0 \\
      0 & 2
    \end{pmatrix}
    \quad\text{and}\quad
    A_4 =
    \begin{pmatrix}
      0 & 1 \\
      1 & 0
    \end{pmatrix}.
  \end{equation*}
  Then $(A_3,A_4)$ is irreducible, has no invariant multicone, and does not contain only conformal matrices. The claim follows now from Proposition \ref{prop:gibbstype} and Theorem \ref{thm:fortuples}. We remark that this phenomenon has been observed earlier in \cite[\S 1.4]{FraserJordanJurga2017}. Another way to see the claim is to consider two conformal irreducible matrices not sharing a conjugation matrix.
\end{example}

\section{Characterization of domination}\label{sec:dom}

In this section, we prove Theorem \ref{thm:justdomin} and Propositions \ref{prop:connect} and \ref{prop:converse}. Let $\mathsf{A} \subset GL_2(\R)$ and recall that $\mathcal{S}(\mathsf{A})$ is the sub-semigroup of $GL_2(\mathbb{R})$ generated by $\mathsf{A}$. Let $\mathscr{S}(\mathsf{A}) = \overline{\mathbb{R}\mathcal{S}(\mathsf{A})}\subset M_2(\mathbb{R})$ and note that $\mathscr{S}(\mathsf{A})$ is a sub-semigroup of $M_2(\mathbb{R})$. Define
\begin{equation*}
  \mathscr{R}(\mathsf{A}) = \{ A \in \mathscr{S}(\mathsf{A}) : \rank(A)=1 \}.
\end{equation*}

\begin{lemma}\label{lem:onlyel}
  If $\A \subset GL_2(\R)$, then $\mathscr{R}(\mathsf{A})=\emptyset$ if and only if $\mathsf{A}$ is strongly conformal.
\end{lemma}

 \begin{proof}
  If $\A$ is strongly conformal, then by definition there exists a conjugation matrix $M \in GL_2(\mathbb{R})$ such that $|\det(A)|^{-1/2} MAM^{-1}\in O(2)$ for all $A \in \A$, which implies that $|\det(A)|^{-1/2} MAM^{-1}\in O(2)$ for all nonzero $A \in \mathscr{S}(\A)=\overline{\R\SS(\A)}$. In particular, all nonzero elements of $ \mathscr{S}(\A)$ have rank $2$ and therefore $\mathscr{R}(\mathsf{A})=\emptyset$.

  Suppose conversely that $\mathscr{R}(\A)=\emptyset$. We claim that set 
  $$
  \mathscr{S}'(\A)=\{|\det(A)|^{-1/2}A : A\in\mathscr{S}(\A)\setminus\{\boldsymbol{0}\}\}=\mathscr{S}(\A)\cap\{A \in M_2(\R) : |\det(A)|=1\}
  $$
  is compact. It is obviously closed, being the intersection of $\mathscr{S}(\A)$ with the closed set $\{A \in M_2(\mathbb{R}) : |\det(A)|=1\}$. If it contains a sequence of elements $(A_n)$ such that $\|A_n\| \to \infty$ then this sequence can without loss of generality be taken to be a sequence of elements of $\mathbb{R}\mathcal{S}(\A)$. The sequence of normalised matrices $\|A_n\|^{-1}A_n$ has an accumulation point which necessarily has determinant zero and norm one and belongs to $\mathscr{S}(\A)$; this limit point is thus an element of $\mathscr{R}(\mathsf{A})$, which is a contradiction, and we conclude that $\{|\det(A)|^{-1/2}A : A\in\mathscr{S}(\A)\setminus\{\boldsymbol{0}\}\}$ is bounded. It is therefore compact as claimed.

  The set $\mathscr{S}'(\A)$ is thus a compact sub-semigroup of $GL_2(\mathbb{R})$. We claim that it is a group. To show this it is sufficient to show that the inverse of every $A\in \mathscr{S}(\A)$ with $|\det(A)|=1$ belongs to $\mathscr{S}(\A)$. If $A \in \mathscr{S}(\A)$ is arbitrary, take a convergent subsequence $(A^{n_k})_{k=1}^\infty$ of the sequence $(A^n)_{n=1}^\infty$ with limit $B \in \mathscr{S}(\A) \subset GL_2(\mathbb{R})$, say.  The sequence $(A^{-n_k})_{k=1}^\infty$ clearly converges to $B^{-1}$ and therefore $A^{n_{k+1} - n_k-1} \to A^{-1}$ as $k \to \infty$. Thus $A^{-1}$ is the accumulation point of a sequence of elements of $\mathscr{S}(\A)$, hence an element of $\mathscr{S}(\A)$.

  The set $\mathscr{S}'(\A)$  is therefore a compact subgroup of $GL_2(\mathbb{R})$. If $m$ is Haar measure on $\mathscr{S}'(\A)$ and $\langle \cdot,\cdot\rangle$ is the standard inner product on $\mathbb{R}^2$ it is easy to see that $\langle u,v\rangle':=\int \langle Au,Av\rangle \dd m(A)$ defines an inner product on $\mathbb{R}^2$ which is invariant under every element of $\mathscr{S}'(\A)$. Every inner product on $\mathbb{R}^2$ is related to the standard one by a change of basis, so there exists $X \in GL_2(\mathbb{R})$ such that $\langle u,v\rangle'=\langle Xu,Xv\rangle$ for all $u,v \in \mathbb{R}^2$. In particular, $\langle XAX^{-1}u,XAX^{-1}v\rangle=\langle AX^{-1}u,AX^{-1}v\rangle'=\langle X^{-1}u,X^{-1}v\rangle'=\langle u,v\rangle$ for all $u,v \in \mathbb{R}^2$ and $A \in \mathscr{S}'(\A)$ which yields $\mathscr{S}'(\A)\subset XO(2)X^{-1}$. Thus $\mathscr{S}(\A)$ is strongly conformal and therefore $\A$ is strongly conformal as required.
 \end{proof}

We note that according to the previous lemma, $\mathscr{R}(\A)\neq\emptyset$ if and only if $\SS(\A)$ contains at least one proximal or parabolic element. In the next lemma, we exclude parabolic elements.

\begin{lemma} \label{thm:dom-lemma0}
  Let $\A \subset GL_2(\R)$ with $\mathscr{R}(\mathsf{A}) \ne \emptyset$ be such that $\mathcal{S}(\mathsf{A})$ is almost multiplicative. Then $\mathcal{S}(\mathsf{A})$ does not contain parabolic elements and $\mathscr{R}(\A)$ does not contain nilpotent elements.
\end{lemma}

\begin{proof}
Suppose that $\mathcal{S}(\A)$ contains a parabolic element. This means that, after a suitable change of basis, there exists $A\in\mathcal{S}(\A)$ such that
  $$
    A=\begin{pmatrix}
    a & 0 \\
    b & a
    \end{pmatrix},
  $$
  where $b\neq 0$. It follows that there exists $c>0$ such that $c^{-1}n|a^{n-1}b|\leq\|A^n\|\leq cn|a^{n-1}b|$ for all $n \in \N$. It follows directly that $\lim_{n \to \infty} \|A^{2n}\|/\|A^n\|^2=0$ which contradicts the condition  $\|AB\|\geq \kappa \|A\|\|B\|$.
  
  Observe that the relation $\|AB\| \ge \kappa\|A\|\|B\|$ holds for all $A,B \in \mathscr{S}(\mathsf{A})$ by continuity. So similarly, if there exists a nilpotent $A\in\mathscr{R}(\A)$, then $0=\|A^n\|\geq\kappa^{n-1}\|A\|^n>0$ for some $n\in\N$, which is again a contradiction.
\end{proof}

Assuming $\mathscr{R}(\mathsf{A}) \ne \emptyset$, we define the set $X_u$ of all \emph{unstable directions} of proximal elements of $\mathcal{S}(\A)$ to be
\begin{equation*}
  X_u = X_u(\A) = \{V \in \mathbb{RP}^1 : V=A\R^2 \text{ for some } A \in \mathscr{R}(\mathsf{A}) \}
\end{equation*}
and the set $X_s$ of all \emph{stable directions} to be
\begin{equation*}
  X_s = X_s(\A) = \{ \ker(A) \in \mathbb{RP}^1 : A \in \mathscr{R}(\mathsf{A}) \}.
\end{equation*}

\begin{lemma} \label{thm:dom-lemma1}
  Let $\A \subset GL_2(\R)$ with $\mathscr{R}(\mathsf{A}) \ne \emptyset$ be such that $\mathcal{S}(\mathsf{A})$ is almost multiplicative. Then the sets $X_u$ and $X_s$ are nonempty, compact, and disjoint. Furthermore, $AX_u \subset X_u$ for all $A \in \mathscr{S}(\mathsf{A})$.\end{lemma}

\begin{proof}
  First, we note again that the relation $\|AB\| \ge \kappa\|A\|\|B\|$ holds for all $A,B \in \mathscr{S}(\mathsf{A})$. To see that $X_u$ and $X_s$ are disjoint, note that if $V \in X_u \cap X_s$ then there exist nonzero $B_1,B_2 \in \mathscr{R}(\mathsf{A})$ such that $B_2\mathbb{R}^2=V$ and $B_1V=\{0\}$. Hence $B_1B_2$ is the zero matrix but $B_1$ and $B_2$ are not, which contradicts $\|B_1B_2\|\geq \kappa \|B_1\|\|B_2\|>0$. It follows that $X_u \cap X_s$ is empty. The nonempty set
  \[
  \mathscr{R}_1(\mathsf{A})=\{B \in \mathscr{R}(\mathsf{A}) : \|B\|=1\}= \{B \in \mathscr{S}(\mathsf{A}) : \det(B)=0 \text{ and }\|B\|=1\}
  \]
  is clearly a closed bounded subset of $\mathscr{S}(\mathsf{A})$, and in particular is compact. It follows that $X_u$ and $X_s$ are the images of continuous functions $\mathscr{R}_1(\mathsf{A}) \to \mathbb{RP}^1$ and hence are compact and nonempty.

  To see the last claim, consider a subspace $U$ such that $U=AV$ for some $V \in X_u$ and $A \in \mathscr{S}(\A)$. We have $V=B\R^2$ for some $B \in \mathscr{R}(\mathsf{A})$. Clearly $AB$ has rank at most $1$ and is nonzero since $\|AB\| \ge \kappa\|A\|\|B\|>0$, so $AB \in \mathscr{R}(\mathsf{A})$ and $U \in X_u$.
\end{proof}

The following lemma shows that the definitions of unstable and stable directions agree with the ones given in \S \ref{sec:subdom}.

\begin{lemma}
  Let $\A \subset GL_2(\R)$ with $\mathscr{R}(\mathsf{A}) \ne \emptyset$ be such that $\mathcal{S}(\mathsf{A})$ is almost multiplicative. Then
  \begin{align*}
    X_u&=\overline{\{u(A): A\in\SS(A)\text{ is proximal}\}}, \\
    X_s&=\overline{\{s(A):A\in\SS(A)\text{ is proximal}\}}.
  \end{align*}
\end{lemma}

\begin{proof}
  Let us first demonstrate the inclusions
  \begin{align}
    X_u &\subset \overline{\{u(A): A\in\SS(A)\text{ is proximal}\}}, \label{eq:eztet} \\
    X_s &\subset \overline{\{s(A): A\in\SS(A)\text{ is proximal}\}}. \label{eq:aztat}
  \end{align}
  Before doing so, we show that for every $A \in \mathscr{R}(\mathsf{A})$ with $\|A\|=1$ and any sequence $(B_n)_{n=1}^\infty$ of elements of $\SS(\A)$ such that $\|B_n\|^{-1}B_n\to A$ as $n\to\infty$, the sequence $B_n$ contains only proximal elements for all sufficiently large $n$. By Lemma~\ref{thm:dom-lemma0}, no $B_n$ may be a parabolic matrix. Let us contrarily assume that, after passing to a suitable subsequence, every $B_n$ is conformal. Write $B_n':=\|B_n\|^{-1}B_n$ for all $n \in \N$. Since $A$ has rank one we have $\det(A)=0$ and therefore $\det(B_n') \to 0$. Since every $B_n'$ is conformal it satisfies $(\tr B_n')^2\leq 4|\det(B_n')|$ and therefore $\tr B_n' \to 0$. By the Cayley-Hamilton theorem, we have $(B_n')^2-(\tr B_n' )B_n'+(\det(B_n'))I=\boldsymbol{0}$ and since $B_n' \to A$ we deduce that $(B_n')^2 \to 0$. Since $\|B_n'\|=1$ for all $n \in \N$ we get $\|B_n^2\|/\|B_n\|^2 = \|(B_n')^2\|/\|B_n'\|^2 =\|(B_n')^2\|\to 0$, but this contradicts $ \|B_n^2 \| \geq \kappa \|B_n\|^2$. We conclude that $(B_n)_{n=1}^\infty$ is proximal for all sufficiently large $n$ as claimed. 
  
  It is well known that the maps $u(\cdot)$ and $s(\cdot)$ are continuous on proximal matrices. Moreover, by Lemma~\ref{thm:dom-lemma0}, every $A\in\mathscr{R}(\A)$ is proximal. Hence, if $V\in X_u$, then there exists a proximal $A\in\mathscr{R}(\A)$ with $\|A\|=1$ such that $V=A\R^2=u(A)$. Moreover, there exists a sequence of proximal matrices $B_n\in\SS(\A)$ such that $\|B_n\|^{-1}B_n\to A$ and thus, by the continuity of $u$, $u(B_n)=u(\|B_n\|^{-1}B_n)\to u(A)=V$, which shows \eqref{eq:eztet}. Similarly, if $V\in X_s$, then there exists a proximal $A\in\mathscr{R}(\A)$ with $\|A\|=1$ such that $V=\ker(A)=s(A)$, and there exists a sequence of proximal matrices $B_n\in\SS(\A)$ such that $\|B_n\|^{-1}B_n\to A$. Applying now the continuity of $s$, we get $s(B_n)=s(\|B_n\|^{-1}B_n)\to s(A)=V$ showing \eqref{eq:aztat}.
  
  To finish the characterization of $X_u$ it is sufficient to show that
  \begin{align}
    X_u &\supset \{u(A): A\in\SS(\A)\text{ is proximal}\}, \label{eq:eztet2} \\
    X_s &\supset \{s(A): A\in\SS(\A)\text{ is proximal}\}, \label{eq:aztat2}
  \end{align}
  since we may then appeal to Lemma \ref{thm:dom-lemma1} and the fact that the sets $X_u$ and $X_s$ are closed. If $V=u(A)$ for some proximal $A\in\SS(\A)$, then $\|A^{n}\|^{-1}A^n\to B$ as $n\to\infty$, where $B\in\mathscr{R}(\A)$ is such that $B\R^2=u(B)=V$. This shows \eqref{eq:eztet2}. Similarly, if $V=s(A)$ for some proximal $A\in\SS(\A)$, then $\|A^{n}\|^{-1}A^n\to B$ as $n\to\infty$, where $B\in\mathscr{R}(\A)$ is such that $\ker(B)=V$. This shows \eqref{eq:aztat2} and completes the proof.
\end{proof}

Let $d$ be the metric on $\mathbb{RP}^1$ defined by taking $d(U,V)$ to be the angle between the subspaces $U$ and $V$. If $\A \subset GL_2(\R)$ is such that $\mathscr{R}(\mathsf{A}) \ne \emptyset$, then we define
\[
  \mathcal{V}_n = \{U\in \mathbb{RP}^1 \colon d(U,V)< \tfrac{1}{n}\text{ for some } V\in X_u \}
\]
and
\[
  \mathcal{U}_n =\bigcup_{A \in \mathcal{S}(\mathsf{A})} A\mathcal{V}_n
\]
for all $n \in \N$.

\begin{lemma} \label{thm:dom-lemma2}
  Let $\A \subset GL_2(\R)$ with $\mathscr{R}(\mathsf{A}) \ne \emptyset$ be such that $\mathcal{S}(\mathsf{A})$ is almost multiplicative. Then there is $n_0 \in \N$ such that $\overline{\mathcal{U}}_n$ as defined above is an invariant unstable multicone for all $n \ge n_0$.
\end{lemma}

\begin{proof}
  Note that for all $n \in \N$ the invariance of $\overline{\mathcal{U}}_n$ and the property \eqref{i-X4} in the definition of the unstable multicone (see \S \ref{sec:subdom}) follow immediately from the definition of the set $\mathcal{U}_n$ and the continuity of each $A \in \mathcal{S}(\mathsf{A})$ as an action on $\RP$. Let us prove the property \eqref{i-X5} for all $n \in \N$. Obviously $\mathcal{V}_n$ is open, and since each $A \in \mathcal{S}(\mathsf{A})$ is invertible and therefore induces a homeomorphism of $\mathbb{RP}^1$, each $\mathcal{U}_n$ is open too. It is clear from the definition that every connected component of $\mathcal{V}_n$ intersects $X_u$. If $U \in \mathcal{U}_n$, then $U=AU'$ for some $A \in \mathcal{S}(\mathsf{A})$ and $U' \in \mathcal{V}_n$. Let $\mathcal{I}\subset \mathcal{V}_n$ be an open connected set which contains $U'$ and which also intersects $X_u$. The set $A\mathcal{I}$ then contains $U$, is connected, and intersects $AX_u$. Since $AX_u \subset X_u$ by Lemma \ref{thm:dom-lemma1}, we conclude that each connected component of $\mathcal{U}_n$ intersects $X_u$.

To show that the property \eqref{i-X3} holds for all large enough $n$, let us suppose the contrary. In this case $\overline{\mathcal{U}}_n \cap X_s$ must be nonempty for infinitely many $n \in \N$. This implies that in any prescribed neighbourhoods of $X_u$ and $X_s$ we may find a subspace $U$ in the neighbourhood of $X_u$ and a matrix $A \in \mathcal{S}(\A)$ such that $AU$ belongs to the neighbourhood of $X_s$. It follows that we may choose a sequence of subspaces $(U_n)$ converging to a limit $U \in X_u$ and a sequence $(A_n)$ of elements of $\mathcal{S}(\A)$ such that $A_nU_n$ converges to a limit $V \in X_s$. Define $B_n:=\|A_n\|^{-1}A_n \in \mathscr{S}(\A)$ for every $n \in \N$, and by passing to a subsequence if necessary we may suppose that $(B_n)$ converges to a limit $B \in \mathscr{S}(\A)$ with norm 1.

We claim that $BU=V$. Let $(u_n)$ be a sequence of unit vectors such that $u_n \in U_n$ for every $n \in \N$ and such that $(u_n)$ converges to a unit vector $u \in U$. It is enough to show that $(B_nu_n)$ converges to $Bu$ and that $Bu$ is nonzero, since we have then shown that $V=\lim_{n \to \infty} B_nU_n=BU$. To see that $Bu$ is nonzero we note that $u \in U \in X_u$ and $B \in \mathscr{S}(\A)$ with $B \neq 0$, so if $Bu=0$ then $u \in \ker B \in X_s$ and we have $U \in X_s \cap X_u$ contradicting Lemma \ref{thm:dom-lemma1}. On the other hand since
\[0\leq \|B_nu_n-Bu\| \leq \|B_nu_n-B_nu\|+\|B_nu-Bu\| \leq \|u_n-u\|+\|B_n-B\| \to 0 \]
we have $B_nu_n \to Bu$ as $n \to \infty$ as required. But the equation $BU=V$ is impossible since $BU \in X_u$ by Lemma \ref{thm:dom-lemma1} and therefore $V \in X_s \cap X_u$, contradicting Lemma \ref{thm:dom-lemma1}. We conclude that $\overline{\mathcal{U}}_n \cap X_s$ must be empty for all large enough $n$ and therefore property \eqref{i-X3} holds for all $n$ sufficiently large. 

  We are left to show that $\overline{\mathcal{U}}_n$ is a multicone. To that end, it suffices to show that $\partial\UU_n$ contains only finitely many points. To see this suppose for a contradiction that $U \in \mathbb{RP}^1$ is an accumulation point of a sequence $(U_k)_{k=1}^\infty$ of distinct elements of $\partial\mathcal{U}_n$. We will find it convenient to identify a small open neighbourhood $\mathcal{I}$ of $U$ with a bounded open interval $(a,b)\subset \mathbb{R}$. By passing to a subsequence if necessary we may assume that $(U_k)_{k=1}^\infty$ is monotone with respect to the natural order on $\mathcal{I}$, and without loss of generality we assume $(U_k)_{k=1}^\infty$ to be strictly increasing.

  We assert that every interval $(U_k,U_{k+2})$ contains a point of $X_u$. Since $U_{k+1}$ is in the closure of $\mathcal{U}_n$, there exists a point of $\mathcal{U}_n$ in the interval $(U_k,U_{k+2})$. Since neither $U_k$ nor $U_{k+2}$ can belong to $\mathcal{U}_n$, it follows that some connected component of $\mathcal{U}_n$ is contained wholly within the interval $(U_k,U_{k+2})$. By \eqref{i-X5}, this implies that a point of $X_u$ must lie in the interval $(U_k,U_{k+2})$. Since this is true for every $k \in \N$, it follows that $U$ is an accumulation point of $X_u$ and hence, by Lemma \ref{thm:dom-lemma1}, $U$ belongs to $X_u$. But $X_u$ is a subset of $\mathcal{U}_n$ and therefore $U \in \mathcal{U}_n$, which implies that $U_k \in \mathcal{U}_n$ for all sufficiently large $k$. This is clearly impossible since no element of $\partial \mathcal{U}_n$ can be an element of $\mathcal{U}_n$. This contradiction proves that $\partial \mathcal{U}_n$ must be finite.
\end{proof}

The above lemmas prove Theorem \ref{thm:justdomin}:

\begin{proof}[Proof of Theorem~\ref{thm:justdomin}]
  If $\mathscr{R}(\mathsf{A}) = \emptyset$, then, by Lemma~\ref{lem:onlyel}, the set $\mathsf{A}$ is strongly conformal. If $\mathscr{R}(\mathsf{A}) \ne \emptyset$, then the claim follows from Lemmas \ref{thm:dom-lemma0} and \ref{thm:dom-lemma2}.
\end{proof}

Let us next turn to the proof of the propositions.

\begin{lemma}\label{lem:elcone}
Let $A\in GL_2(\R)$ and let $\CC$ be a multicone such that $A\CC\subset\CC$. If $A$ is conformal, then $A\CC=\CC$.
\end{lemma}

\begin{proof}
  By a suitable change of basis, we may assume that $A\in O(2)$. In this case $A$, preserves Lebesgue measure on $\mathbb{RP}^1$. If $A\CC\subsetneq\CC$, then, since $\CC$ is a finite union of closed projective intervals and $A$ is a homeomorphism, $A\CC$ must have smaller Lebesgue measure than $\CC$, which is a contradiction.
\end{proof}

We remark that the converse statement is false: if $A$ is proximal and $\CC$ is a closed projective interval with one endpoint equal to $u(A)$ and the other endpoint equal to $s(A)$, then $A\CC=\CC$ but $A$ is not conformal.

If $\A\subset GL_2(\R)$ and $\A_e$ is the collection of all conformal elements of $\A$, then we write
\begin{equation*}
  \mathcal{F}(\A):=\mathcal{S}(\{|\det(A)|^{-1/2}A : A \in \A_{e}\}).
\end{equation*}

\begin{lemma} \label{thm:dom-lemma3}
  Let $\A\subset GL_2(\R)$ be such that $\mathscr{R}(\A)\neq\emptyset$ and $\A_e$ be the collection of all conformal elements of $\A$. If $\CC$ is an invariant unstable multicone of $\A$, then $\A_e = \{A \in \A : A\CC = \CC\}$ is strongly conformal and $\mathcal{F}(\A)$ is finite.
\end{lemma}

\begin{proof}
  Since $\A$ has an invariant multicone $\CC$, it follows from Lemma~\ref{lem:elcone} that $A\CC=\CC$ for all $A\in\A_e$. Hence $\A_e \subset \{A\in\A: A\CC=\CC\}$.

  Write $\A_{e}' = \{|\det(A)|^{-1/2}A : A\in\A\text{ and }A\CC=\CC\}$. Let us first assume that $\#\partial \CC>2$. Let $B_1,B_2 \in \mathcal{S}(\A_{e}')$ and suppose that $B_1$ and $B_2$ induce the same permutation of $\partial\mathcal{C}$. Then $B_1^{-1}B_2$ fixes every point of $\partial\mathcal{C}$ and therefore has more than $2$ invariant subspaces and is necessarily equal to $\pm I$. It follows that in this case $\mathcal{S}(\A_{e}')$ has at most $2(\#\partial\mathcal{C})!$ distinct elements. Let us now assume that $\#\partial \mathcal{C}=2$. Write $\partial \mathcal{C}=\{U_1,U_2\}$, and let $u_1 \in U_1$ and $u_2 \in U_2$ be so that $\{u_1,u_2\}$ is a basis for $\mathbb{R}^2$. Every element of $\mathcal{S}(\A_{e}')$ preserves $\partial\mathcal{C}$ and hence is either diagonal or antidiagonal in this basis (where by an antidiagonal matrix we mean a  $2\times 2$ matrix with both main diagonal entries equal to zero and both other entries nonzero). Let $D$ be the matrix which $Du_1=u_1$ and $Du_2=-u_2$. A diagonal element of $\mathcal{S}(\A_{e}')$ cannot be proximal since then either $U_1$ or $U_2$ would be the stable space of that matrix contradicting the property $X_s \cap \mathcal{C} =\emptyset$ of the unstable multicone $\CC$. It follows that every diagonal element of $\mathcal{S}(\A_{e}')$ must belong to $\{\pm I,\pm D\}$. Let $A_1,\ldots,A_\ell$ be the anti-diagonal elements of $\A_{e}'$ and define $S=\{\pm I, \pm D\} \cup \{\pm A_1,\ldots, \pm A_\ell\} \cup \{\pm DA_1,\ldots,\pm DA_\ell\}$. The set $S$ is a semigroup since $A_iD=-DA_i$ and since each $A_iA_j$ is diagonal and hence equal to $\pm I$ or $\pm D$. In particular, $\mathcal{S}(\A_{e}')$ is contained in a finite semigroup. Thus, $\A_{e}'$ is strongly conformal, which implies that $\{A\in\A: A\CC=\CC\}\subset\A_{e}$.
\end{proof}

\begin{lemma} \label{thm:dom-lemma4}
  Let $\A\subset GL_2(\R)$ be such that $\A$ has an invariant unstable multicone $\CC$ and $\SS(\A)$ does not contain parabolic elements. Let $\A_{e}$ be the collection of all conformal elements of $\A$. Then
  $$
    A_1F_1\cdots A_nF_n\CC \subset \CC^o
  $$
  for all $n\geq (\#\partial\CC)^2+1$, $A_1,\ldots,A_n \in \mathsf{A} \setminus \A_{e}$, and $F_1,\ldots,F_n \in \mathcal{F}(\A)$.
\end{lemma}

\begin{proof}
  It is sufficient to show that every point of $\partial\mathcal{C}$ is mapped into $\mathcal{C}^\circ$ by $A_1F_1\cdots A_nF_n$. Clearly, if there exists $\ell \in \{1,\ldots,n\}$ such that $A_\ell F_\ell\cdots A_nF_n\CC\subset\CC^o$, then our claim follows.

  Suppose for a contradiction that there exist $n\geq (\#\partial\CC)^2+1$, $A_1,\ldots,A_n\in\A\setminus\A_{e}$, and $F_1,\ldots,F_n\in\mathcal{F}(\A)$ such that for every $\ell \in \{1,\ldots,n\}$ there exist $V_\ell,W_\ell\in\partial\CC$ for which
  $$
    A_\ell F_\ell\cdots A_nF_nV_{\ell}=W_{\ell}.
  $$
  Since $n\geq (\#\partial\CC)^2+1$, there exist $\ell_1<\ell_2$ such that $V_{\ell_1}=V_{\ell_2}$ and $W_{\ell_1}=W_{\ell_2}$. Hence,
  $$
    A_{\ell_1}F_{\ell_1}\cdots A_{\ell_2-1}F_{\ell_2-1}W_{\ell_2}=W_{\ell_2}.
  $$
  Thus, if $A_{\ell_1}F_{\ell_1}\cdots A_{\ell_2-1}F_{\ell_2-1}$ is proximal, then $W_{\ell_2}\in X_u\cup X_s$. This is impossible, since $\partial\CC\cap (X_s\cup X_u)=\emptyset.$ If $A_{\ell_1}F_{\ell_1}\cdots A_{\ell_2-1}F_{\ell_2-1}$ is conformal, then $A_{\ell_1}F_{\ell_1}\cdots A_{\ell_2-1}F_{\ell_2-1}\CC=\CC$ by Lemma~\ref{lem:elcone}. This is also impossible, since $\CC\supsetneq A_{\ell_1}\CC\supset A_{\ell_1}F_{\ell_1}\cdots A_{\ell_2-1}F_{\ell_2-1}\CC$ by Lemma~\ref{thm:dom-lemma3}.
\end{proof}

\begin{lemma}\label{lem:fin}
  Let $\A\subset GL_2(\R)$ be such that $\A$ has an invariant unstable multicone $\CC$ and $\SS(\A)$ does not contain parabolic elements. Let $\A_{e}$ be the collection of all conformal elements of $\A$. If $\A \setminus \A_e$ is compact, then
  $$
    \B=\{A_1A_2\colon A_1\in\A\setminus\A_{e}\text{ and }A_2\in\mathcal{F}(\A)\}
  $$
  has a strongly invariant multicone.
\end{lemma}

\begin{proof}
  Write $m=(\#\partial\CC)^2+1$ and note that, by Lemma~\ref{thm:dom-lemma4}, $\B^{m}$ has a strongly invariant multicone. Since $\A\setminus\A_{e}$ is compact by the assumption and $\mathcal{F}(\A)$ is finite by Lemma \ref{thm:dom-lemma3}, $\B^m$ is compact. Hence, by \cite[Theorem~B]{BochiGourmelon09}, $\B^m$ is dominated, i.e.\ there exist constants $C>0$ and $\tau>1$ such that
  $$
    \frac{\|B_1\cdots B_n\|}{\|(B_1\cdots B_n)^{-1}\|^{-1}}\geq C\tau^n.
  $$
  for all $B_1,\ldots,B_n\in\B^m$ and all $n\in\N$. Choose $k \in \N$ and let $A_iF_i\in\B$ for all $i \in \{1,\ldots,k\}$. Write $k=qm+p$, where $q\in\N\cup\{0\}$ and $p \in \{0,\ldots,m-1\}$. Then
  \begin{align*}
    \frac{\|A_1F_1\cdots A_kF_k\|}{\|(A_1F_1\cdots A_kF_k)^{-1}\|^{-1}}
    & = \dfrac{\|B_1\cdots B_q\cdot A_{k-p}F_{k-p}\cdots A_kF_k\|}{\|(B_1\cdots B_q\cdot A_{k-p}F_{k-p}\cdots A_kF_k)^{-1}\|^{-1}}\\
    & \geq \dfrac{\|B_1\cdots B_q\|\|(A_{k-p}F_{k-p}\cdots A_kF_k)^{-1}\|^{-1}}{\|(B_1\cdots B_q)^{-1}\|^{-1}\|A_{k-p}F_{k-p}\cdots A_kF_k\|}\\
    & \geq C\tau^{q}\dfrac{\|(A_{k-p}F_{k-p}\cdots A_kF_k)^{-1}\|^{-1}}{\|A_{k-p}F_{k-p}\cdots A_kF_k\|}.
  \end{align*}
  By choosing $C'=C\tau^{-1}\min_{\ell\in\{1,\ldots,m-1\}} \|(A_{1}F_{1}\cdots A_\ell F_\ell)^{-1}\|^{-1}/\|A_{1}F_{1}\cdots A_\ell F_\ell\|$ and $\tau'=\tau^{1/m}$, it follows again from \cite[Theorem~B]{BochiGourmelon09} that $\B$ has a strongly invariant multicone.
\end{proof}

The following lemma is \cite[Lemma~2.2]{BochiMorris15}.

\begin{lemma}\label{lem:Morris}
  Let $\CC_0,\CC\subset\RP$ be multicones such that $\CC_0\subset\CC^o$. Then there exists a constant $\kappa_0>0$ such that $\|A|V\|\geq \kappa_0\|A\|$ for all $V\in\CC_0$ and for every matrix $A\in GL_2(\R)$ with $A\CC\subset\CC_0$.
\end{lemma}

We are now ready to prove the propositions:

\begin{proof}[Proof of Proposition~\ref{prop:connect}]
  The assertion \eqref{item:this} follows immediately from Lemma~\ref{thm:dom-lemma3}. Let us verify \eqref{item:that}. If $\A_{e} = \A$, then $\SS(\A)$ is strongly conformal since $\A$ is. This means that $\SS(\A)$ does not contain an proximal matrix and thus, $\A$ cannot have an unstable multicone by definition. Therefore, \eqref{item:that} holds.

  To prove the final claim, it is sufficient to show that, by assuming $\A\setminus\A_e$ to be compact, there exists an invariant multicone $\CC$ such that $A\CC\subset\CC^o$ for all $A\in\A\setminus\A_{e}$ and $A\CC=\CC$ for all $A\in\A_{e}$. By Lemma~\ref{thm:dom-lemma3}, the set $\mathcal{F}(\A)$ is finite. Therefore, the set $\B=\{A_1A_2 : A_1\in\A\setminus\A_{e}\text{ and }A_2\in\mathcal{F}(\A)\}$ is compact and, by Lemma~\ref{lem:fin}, it has a strongly invariant multicone $\CC_0$. Defining
  $$
    \CC=\bigcup_{F\in\mathcal{F}(\A)}F\CC_0,
  $$
  we have
  $$
    A\CC = \bigcup_{F\in\mathcal{F}(\A)} AF\CC_0 \subset\CC_0^o\subset\CC^o.
  $$
  for all $A\in\A\setminus\A_{e}$. We have finished the proof since for any $A\in\A_{e}$, $A\CC=\CC$ holds trivially.
\end{proof}

\begin{proof}[Proof of Proposition~\ref{prop:converse}]
  Let $\eps>0$ and define
  $$
    \CC_0=\bigcup_{F\in\mathcal{F}(\A)} F\biggl( \biggl\{ U \in \RP : d(U,V) \le \eps \text{ for some } V \in \bigcup_{A\in\A_{h}}A\CC \biggr\} \biggr)
  $$
  Recall that $\mathcal{F}(\A)$ is finite by Lemma~\ref{thm:dom-lemma3}. By compactness of $\A_h$, we may choose $\eps>0$ small enough so that $\CC_0\subset\CC^o$, $A\CC\subset\CC_0$ for all $A\in\A_{h}$, and $A\CC_0=\CC_0$ for all $A\in\A_{e}$. Observe that every element $A\in\mathcal{S}(\A_h \cup \A_e)$ can be written in the form $(c_0c_1\cdots c_k)F_0\prod_{i=1}^kA_iF_i$, where $c_i\in\R\setminus\{0\}$, $k\in\N\cup\{0\}$, $A_i\in\A_{h}$, and $F_i\in\mathcal{F}(\A)$. Therefore, $A\CC\subset\CC_0$ for all $A\in\mathcal{S}(\A_{h}\cup\A_{e})\setminus\mathcal{S}(\A_{e})$.

  By Lemma~\ref{lem:Morris}, there exists a constant $\kappa_0=\kappa_0(\CC_0,\CC)$ such that $\|A|V\|\geq\kappa_0\|A\|$ for all $V\in\CC_0$ and for every matrix $A \in GL_2(\R)$ with $A\CC\subset\CC_0$. Hence,
  $$
    \|AB\|\geq\|AB|V\|=\|A|BV\|\|B|V\|\geq\kappa_0^2\|A\|\|B\|.
  $$
  for all $A,B\in\mathcal{S}(\A_{h}\cup\A_{e})\setminus\mathcal{S}(\A_{e})$. If $A\in\mathcal{S}(\A_{e})$ or $B\in\mathcal{S}(\A_{e})$, then $\|AB\|\geq \kappa'\|A\|\|B\|$ holds trivially for some $\kappa'>0$ by the finiteness of $\mathcal{F}(\A)$.
\end{proof}

\section{Classification of equilibrium states}

This section is devoted to the proofs of Propositions \ref{prop:gibbstype} and \ref{prop:bernoulli}, and Theorems \ref{thm:fortuples} and \ref{thm:holder}. In order to keep the proof of Theorem \ref{thm:holder} as readable as possible, we have postponed the proof of a key technical lemma, Lemma \ref{prop:key}, into \S \ref{sec:3m}. Before we start with the proof of the propositions, we state a couple of auxiliary lemmas.

We recall that $\lambda_u(A)$ is the eigenvalue of $A$ with the largest absolute value, and similarly, $\lambda_s(A)$ is the eigenvalue of $A$ with the smallest absolute value. Note that $|\lambda_u(A)|=\|A|u(A)\|$ and $|\lambda_s(A)|=\|A|s(A)\|$, where $u(A)$ is the eigenspace corresponding to $\lambda_u(A)$ and $s(A)$ the eigenspace corresponding to $\lambda_s(A)$.

The following two lemmas are special cases of the result of Protasov and Voynov; see \cite[Theorem~2]{ProtVoy}. In order to keep the paper as self-contained as possible, we give here alternative proofs.

\begin{lemma}\label{lem:eigenspace}
  Let $\A=(A_1,\ldots,A_N)\in GL_2(\R)^N$ be such that all the elements of $\A$ are proximal. Then the following two statements are equivalent:
  \begin{enumerate}
    \item\label{it:eig1} $\lambda_u(A_iA_j)=\lambda_u(A_i)\lambda_u(A_j)$ for all $i,j$,
    \item\label{it:eig2} $u(A_i)=u(A_j)$ for all $i,j$ or $s(A_i)=s(A_j)$ for all $i,j$.
  \end{enumerate}
\end{lemma}

\begin{proof}
  It is easy to see that \eqref{it:eig2} implies \eqref{it:eig1}. Let us show that \eqref{it:eig1} implies \eqref{it:eig2}. By the assumption and the multiplicativity of the determinant, we have $\lambda_s(A_iA_j)=\lambda_s(A_i)\lambda_s(A_j)$ for all $i,j$. First note that $s(A_i)\neq u(A_j)$, for any $i\neq j$. Indeed, $s(A_i)= u(A_j)$ would imply that the matrix $A_iA_j$ has eigenvalue $\lambda_s(A_i)\lambda_u(A_j)$. Thus, either $\lambda_s(A_i)\lambda_u(A_j)=\lambda_u(A_i)\lambda_u(A_j)$ or $\lambda_s(A_i)\lambda_u(A_j)=\lambda_s(A_i)\lambda_s(A_j)$, which implies that either $\lambda_s(A_i)=\lambda_u(A_i)$ or $\lambda_u(A_j)=\lambda_s(A_j)$, which contradicts to the proximality.

  We prove the statement by induction. Since $s(A_1)\neq u(A_2)$, after a suitable change of basis, the matrices $A_1$ and $A_2$ have the form
  $$
  A_1=\begin{pmatrix}
        \lambda_u(A_1) & 0 \\
        a & \lambda_s(A_1)
      \end{pmatrix}\quad\text{and}\quad A_1=\begin{pmatrix}
        \lambda_u(A_1) & b \\
       0 & \lambda_s(A_1)
      \end{pmatrix}.
  $$
  Hence, $\tr(A_1A_2)=\lambda_u(A_1A_2)+\lambda_s(A_1A_2)=\lambda_u(A_1)\lambda_u(A_2)+\lambda_s(A_1)\lambda_s(A_2)+ab$. So $ab=0$, which implies that if $b=0$ then $s(A_1)=s(A_2)$ or if $a=0$ then $u(A_1)=u(A_2)$.

  Let us then assume that the first $N-1$ matrices have the property that either $u(A_i)=u(A_j)$ for all $i,j\in\{1,\ldots,N-1\}$ or $s(A_i)=s(A_j)$ for all $i,j\in\{1,\ldots,N-1\}$. We may assume without loss of generality that $u(A_i)=u(A_j)$ for all $i,j\in\{1,\ldots,N-1\}$. For a fixed $i\in\{1,\ldots,N-1\}$ the equation $\lambda_u(A_i)\lambda_u(A_N)=\lambda_u(A_iA_N)$ holds only if $u(A_i)=u(A_N)$ or $s(A_i)=s(A_N)$. If $u(A_i)=u(A_N)$ for some $i\in\{1,\dots,N-1\}$, then the proof is complete; otherwise $s(A_i)=s(A_N)$ must hold for all $i\in\{1,\dots,N-1\}$, which again implies the claimed property.
\end{proof}

\begin{lemma}\label{lem:eigenspace2}
  Let $\A=(A_1,\ldots,A_N)\in GL_2(\R)^N$ be such that all the elements of $\A$ are proximal. The following two statements are equivalent:
  \begin{enumerate}
    \item\label{it:eig1b} $|\lambda_u(AB)|=|\lambda_u(A)\lambda_u(B)|$ for all $A,B\in\mathcal{S}(\A)$,
    \item\label{it:eig2b} $u(A_i)=u(A_j)$ for all $i,j$ or $s(A_i)=s(A_j)$ for all $i,j$.
  \end{enumerate}
\end{lemma}

\begin{proof}
It is again easy to see that \eqref{it:eig2b} implies \eqref{it:eig1b}. Therefore, we assume that \eqref{it:eig1b} holds. Let us first show that $\lambda_u(A_iA_j)=\lambda_u(A_i)\lambda_u(A_j)$ or $\lambda_u(A_iA_j^2)=\lambda_u(A_i)\lambda_u(A_j)^2$ for every $i\neq j$. Suppose for a contradiction that there exist $i \ne j$ such that
\[
\lambda_u(A_iA_j)=-\lambda_u(A_i)\lambda_u(A_j)\quad\text{and}\quad\lambda_u(A_iA_j^2)=-\lambda_u(A_i)\lambda_u(A_j)^2.
\]
Hence, $\lambda_u(A_iA_j)\lambda_u(A_j)=-\lambda_u(A_i)\lambda_u(A_j)^2=\lambda_u(A_iA_j^2)$ and, by Lemma~\ref{lem:eigenspace} applied to the matrix pair $(A_iA_j,A_j)$, we have $u(A_iA_j)=u(A_j)$ or $s(A_iA_j)=s(A_j)$. Assuming $u(A_iA_j)=u(A_j)$, we have $-\lambda_u(A_i)\lambda_u(A_j)^2v(A_j)=A_jA_j^2v(A_j)=\lambda_u(A_j)^2A_iv(A_j)$, where $v(A_j)\in u(A_j)$ is a unit vector. But this is a contradiction since this would imply that $\lambda_u(A_i)=0$ or $\lambda_s(A_i)=-\lambda_u(A_i)$. The case $s(A_iA_j)=s(A_j)$ is similar.

If $\lambda_u(A_iA_j)=\lambda_u(A_i)\lambda_u(A_j)$, then \eqref{it:eig2b} follows from Lemma~\ref{lem:eigenspace}. Similarly, if $\lambda_u(A_iA_j^2)=\lambda_u(A_i)\lambda_u(A_j)^2$, then again by Lemma~\ref{lem:eigenspace}, $u(A_j)=u(A_j^2)=u(A_i)$ or $s(A_j)=s(A_j^2)=s(A_i)$. The proof can be finished by induction similarly to the proof of Lemma~\ref{lem:eigenspace}.
\end{proof}

The following lemma is a simple application of \cite[Theorem~1.7(ii)--(iii)]{FengKaenmaki2011}.

\begin{lemma}\label{lem:triang}
  Let $\A=(A_1,\ldots,A_N)\in GL_2(\R)^N$ be such that
  $$
  A_i=
  \begin{pmatrix}
    a_i & b_i \\
    0 & c_i
  \end{pmatrix}
  $$
  for all $i \in \{1,\ldots,N\}$, where $a_i,b_i,c_i \in \R$, and let $\mu_a$ and $\mu_c$ be the Bernoulli measures obtained from the probability vectors $(\sum_{i=1}^N|a_i|^s)^{-1}(|a_1|^s,\ldots,|a_N|^s)$ and $(\sum_{i=1}^N|c_i|^s)^{-1}(|c_1|^s,\ldots,|c_N|^s)$, respectively. If $\mu$ is an ergodic equilibrium state for $\Phi^s$, then
  \begin{equation*}
    \mu \in
    \begin{cases}
      \{\mu_a\}, &\text{if } \sum_{i=1}^N|a_i|^s>\sum_{i=1}^N|c_i|^s, \\
      \{\mu_c\}, &\text{if } \sum_{i=1}^N|a_i|^s<\sum_{i=1}^N|c_i|^s, \\
      \{\mu_a,\mu_c\}, &\text{if } \sum_{i=1}^N|a_i|^s=\sum_{i=1}^N|c_i|^s.
    \end{cases}
  \end{equation*}
\end{lemma}

The following lemma is \cite[Proposition 1.2]{FengKaenmaki2011}.

\begin{lemma}\label{lem:irred}
If $\A\in GL_2(\R)^N$ is irreducible, then there is unique equilibrium state which is a Gibbs-type measure for $\Phi^s$.
\end{lemma}

We are now ready to prove the propositions.

\begin{proof}[Proof of Proposition~\ref{prop:gibbstype}]
Let us first show that \eqref{it:gibbstype2} implies \eqref{it:gibbstype1}. Lemma~\ref{lem:irred} shows that if $\A$ is irreducible then the equilibrium state is a Gibbs-type measure for $\Phi^s$. Also, if $\A$ is strongly conformal, the conclusion is straightforward. We may thus assume that $\A$ is reducible with a common invariant subspace $V$ and that there exists $\varepsilon>0$ such that either the closed $\varepsilon$-neighbourhood of $V$ or the closure of its complement is an invariant unstable multicone. Note that $\SS(\A)$ cannot contain any parabolic elements, since in this case the neighbourhood (or its complement) cannot be invariant.

We may, by Proposition~\ref{prop:connect}, assume that for some $M \in \N$ the tuple $\A_{h}=(A_1,\ldots,A_M)$ has a strongly invariant multicone $\CC$ and $\A_{e}=(A_{M+1},\ldots,A_N)$ is such that $A_i\CC=\CC$ for all $i \in \{M+1,\ldots,N\}$. Thus, either $V\in\CC^o$ or $V\notin\CC$. If $V\in\CC^o$, then $u(A_i)=V$ for all $i\in\{1,\ldots,M\}$ and if $V\notin\CC$, then $s(A_i)=V$ for all $i\in\{1,\ldots,M\}$. By the invariance of $V$ and since $\SS(\A)$ does not contain parabolic element, every $A\in\SS(\A)$ is diagonalisable. So in the first case, for any $A_{i_1},\ldots,A_{i_n}\in\A$,
$$
|\lambda_u(A_{i_1}\cdots A_{i_n})|=\|A_{i_1}\cdots A_{i_n}|V\|=\prod_{\ell=1}^n\|A_{i_\ell}|V\|=\prod_{\ell=1}^n|\lambda_u(A_{i_\ell})|.
$$
In the second case similarly, $|\lambda_s(A_{i_1}\cdots A_{i_n})|=\prod_{\ell=1}^n|\lambda_s(A_{i_\ell})|$, but by the multiplicity of the determinant $|\lambda_u(A_{i_1}\cdots A_{i_n})|=\prod_{\ell=1}^n|\lambda_u(A_{i_\ell})|$. Moreover, by Lemma~\ref{lem:Morris}, there exists a constant $C>0$ such that for every $A\in\SS(\A)\setminus\SS(\A_e)$
$$
  |\lambda_u(A)|\leq\|A\|\leq C|\lambda_u(A)|,
$$
and $|\lambda_u(A)|=\|A\|$ for $A\in\SS(\A_e)$ trivially. Hence, the Bernoulli measure $\lambda$ obtained from the probability vector
$$
  \biggl(\frac{|\lambda_u(A_1)|^s}{\sum_{i=1}^N |\lambda_u(A_i)|^s},\ldots,\frac{|\lambda_u(A_N)|^s}{\sum_{i=1}^N |\lambda_u(A_i)|^s}\biggr)
$$
is a $\sigma$-invariant Gibbs-type measure for $\Phi^s$. Therefore, $\mu=\lambda$.

Let us then show that \eqref{it:gibbstype1} implies \eqref{it:gibbstype2}. We may assume without loss of generality that $\A$ is reducible with common subspace $V$. Moreover, let us assume that neither any $\varepsilon$-neighbourhood of $V$ nor the closures of the complements are invariant unstable multicone. Our goal is to show that the only remaining possibility, $\A$ is strongly conformal, holds.

By reducibility, after a change of basis, every $A_{\iii}\in\SS(\A)$ has the form
$$
  A_\iii=
  \begin{pmatrix}
    a_\iii & b_\iii \\
    0 & c_\iii
  \end{pmatrix},
$$
where $a_{\iii}=\prod_{k=1}^{|\iii|}a_{i_k}$ and $c_{\iii}=\prod_{k=1}^{|\iii|}c_{i_k}$ with some $a_i,b_i,c_i \in \R$ for $i \in \{1,\ldots,N\}$. Then, by Lemma~\ref{lem:triang}, $\mu=\mu_a$ or $\mu=\mu_c$, where $\mu_a$ and $\mu_c$ are defined in the formulation of Lemma~\ref{lem:triang}. If one of the matrices, say $A_\iii\in\SS(\A)$, is parabolic, then $a_\iii=c_\iii$ and $b_\iii\neq0$. It follows that there exists $c>0$ such that $c^{-1}na_\iii^{n-1}b_\iii\leq\|A_\iii^n\|\leq cna_\iii^{n-1}b_\iii$ for all $n \in \N$. By \cite[Theorem~1.7(ii)]{FengKaenmaki2011}, we may assume that $P(\Phi^s) = \log\sum_{i=1}^N |a_i|^s$ and that $\mu = \mu_a$. The definition of $\mu_a$ thus implies that
$$
  C^{-1}\frac{|a_\iii|^{ns}}{n^s|a_\iii|^{s(n-1)}|b_\iii|^s}\leq \frac{\mu([\iii^n])}{\|A_{\iii}^n\|^s\exp(-nP(\Phi^s))}\leq C\frac{|a_\iii|^{ns}}{n^s|a_\iii|^{s(n-1)}|b_\iii|^s}
$$
for all $n \in \N$.
This is a contradiction since $\mu$ was assumed to be a Gibbs-type measure for $\Phi^s$. Thus, $\SS(\A)$ does not contain any parabolic element.

The common subspace $V$ and the fact that $\SS(\A)$ does not contain parabolic element implies that all the matrices in $\A$ are diagonalisable. Since neither any $\varepsilon$-neighbourhood of $V$ nor the closures of the complements are invariant unstable multicones, then either $|a_k|=|c_k|$ and $b_k=0$ for every $k \in \{1,\ldots,N\}$ (which implies that $\A$ is strongly conformal) or there exist $i\neq j$ such that $|a_i|<|c_i|$ and $|a_j|>|c_j|$. If $\mu=\mu_a$, then
$$
  C^{-1}<\frac{\mu([i^n])}{\|A_{i}^n\|^s\exp(-nP(\Phi^s))}\leq C'\frac{|a_i|^{sn}}{|c_i|^{sn}}
$$
for all $n \in \N$,
and similarly, if $\mu=\mu_c$, then
$$
  C^{-1}<\frac{\mu([j^n])}{\|A_{j}^n\|^s\exp(-nP(\Phi^s))}\leq C'\frac{|c_j|^{sn}}{|a_j|^{sn}}
$$
for all $n \in \N$.
Since both inequalities lead to a contradiction, it follows that $\A$ must be strongly conformal.
\end{proof}

\begin{proof}[Proof of Proposition~\ref{prop:bernoulli}]
  Let us first show that \eqref{it:bernoulli2} implies \eqref{it:bernoulli1}. If $\A$ is reducible, then the statement follows directly from Lemma~\ref{lem:triang}. If $\A$ is strongly conformal, then the statement is straightforward.

  Let us then show that \eqref{it:bernoulli1} implies \eqref{it:bernoulli2}. Let us contrarily assume that $\mu$ is a Bernoulli measure, $\A$ is irreducible, and not strongly conformal. By Lemma~\ref{lem:irred}, $\mu$ is a Gibbs-type measure for $\Phi^s$, that is, there exists a constant $C>0$ such that
  \begin{equation}\label{eq:conttobern}
    C^{-1}\leq\frac{\mu([\iii])}{\|A_{\iii}\|^s\exp(-nP(\Phi^s))}\leq C
  \end{equation}
  for all $\iii\in\Sigma_n$ and $n \in \N$. Since $\mu$ is a Bernoulli measure and $\A$ is not strongly conformal, Theorem~\ref{thm:justdomin} implies that $\A$ has an invariant unstable multicone $\CC$ and $\SS(\A)$ does not contain any parabolic element. We may, by Proposition~\ref{prop:connect}, assume that for some $M \in \N$ the tuple $\A_{h}=(A_1,\ldots,A_M)$ has a strongly invariant multicone $\CC$ and $\A_{e}=(A_{M+1},\ldots,A_N)$ is strongly conformal with $A_i\CC=\CC$ for all $i \in \{M+1,\ldots,N\}$.

  By \eqref{eq:conttobern} and the Bernoulli property of $\mu$,
  $$
    C^{-1/n}\leq\frac{\mu([\iii])}{\|A_{\iii}^n\|^{s/n}\exp(-|\iii|P(\Phi^s))}\leq C^{1/n}
  $$
  for all $\iii\in\Sigma_*$ and $n\in \N$. Thus, by letting $n \to \infty$, we see that
  $$
    |\lambda_u(A_{\iii})|=\mu([\iii])^{1/s}\exp(|\iii|P/s).
  $$
  for all $\iii\in\Sigma_*\setminus\bigcup_{k\in\N}\{M+1,\ldots,N\}^k$. Since $\mu$ is a Bernoulli measure, we see that $|\lambda_u(A_{\iii\jjj})|=|\lambda_u(A_{\iii})\lambda_u(A_{\jjj})|$ for any two $\iii,\jjj\in\Sigma_*\bigcup_{k\in\N}\setminus\{M+1,\ldots,N\}^k$.
  Thus Lemma~\ref{lem:eigenspace2} implies that there exists a subspace $V$ such that $u(A_{\iii})=V$ for all $\iii\in\Sigma_*\setminus\bigcup_{k \in \N}\{M+1,\ldots,N\}^k$ or $s(A_{\iii})=V$ for all $\iii\in\Sigma_*\setminus\bigcup_{k \in \N}\{M+1,\ldots,N\}^k$. Without loss of generality, we may assume that we are in the first case.

  Since $|\lambda_u(A_\iii^2A_j)|=|\lambda_u(A_\iii)\lambda_u(A_\iii A_j)|$ and $|\lambda_u(A_\iii^3A_j)|=|\lambda_u(A_\iii)^2\lambda_u(A_\iii A_j)|$, for every $j\in\{M+1,\ldots,N\}$, we have by Lemma~\ref{lem:eigenspace} that $u(A_{\iii}^kA_j)=u(A_{\iii})$, where $k=1$ or $k=2$. Therefore $A_{\iii}^kA_ju(A_{\iii}^kA_j)=A_{\iii}^ku(A_{\iii})$, which implies that $A_jV=V$. Thus, $V$ is an invariant subspace for $\A$. This contradicts the irreducibility assumption.
\end{proof}

Let us next prove the theorems. For the existence of the function in the statement \eqref{thm:ftvii} of Theorem~\ref{thm:fortuples} we need the following lemma.

\begin{lemma}\label{lem:convmeas}
  Let $\A \subset GL_2(\R)$ be a finite set such that $\A=\A_h\cup\A_e$, where $\A_e$ is strongly conformal and $\A_h\neq\emptyset$ has a strongly invariant multicone $\CC$ such that $A\CC=\CC$ for all $A\in\A_e$. Let $m$ be the Haar measure generated by $\A_e$ normalised on $\CC$. Then for every $\iii\in\Sigma$ there exists a probability measure $\nu_{\iii}$ on $\CC$ such that
  $$
    \nu_{\iii}=\lim_{n\to\infty}(A_{\iii|_n})_*m.
  $$
  In particular, $(A_j)_*\nu_{\iii} = \nu_{j\iii}$.
\end{lemma}

\begin{proof}
  Write $\A_h=\{A_1,\ldots,A_M\}$ and $\A_e=\{A_{M+1},\ldots,A_N\}$. Let us divide $\Sigma$ into two disjoint sets
  \begin{align}
    \hat\Sigma&=\{i_1i_2\cdots\in\Sigma: i_n\in\{1,\ldots,M\}\text{ for infinitely many }n\in\N\},\label{eq:decombsigma1}\\
    \Upsilon&=\{i_1i_2\cdots\in\Sigma: \text{ there is }n_0\in\N\text{ such that }i_n\in\{M+1,\ldots,N\}\text{ for all }n > n_0\}.\label{eq:decombsigma2}
  \end{align}
  Fix $\iii \in \hat\Sigma$. By the definition, $\iii\in\hat\Sigma$ can be written as $\iii = \iii_1j_1\iii_2j_2\cdots$, where $$\iii_k \in \bigcup_{n\in\N}\{M+1,\ldots,N\}^n \cup \{\varnothing\}$$ and $i_k \in \{1,\ldots,M\}$ for all $k \in \N$.
  Thus, $A_{\iii_kj_k}\CC\subset\CC^o$ for every $k\in\N$, and there exists a unique $V=V(\iii)\in\RP$ such that $V=\bigcap_{n=0}^{\infty}A_{\iii_1j_1}\cdots A_{\iii_nj_n}(\CC)$.

  Let $g\colon\RP\to\R$ be a continuous function. Since $\RP$ is compact, for every $\varepsilon>0$ there exists $r>0$ such that for every $V,W\in$ with $d(V,W)<r$, $|g(V)-g(W)|<\varepsilon$. Thus, by choosing $n$ sufficiently large so that $\diam(A_{\iii_1j_1}\cdots A_{\iii_{n}j_{n}}(\CC))<r$, we have
  $$
  \biggl|\int g(V)\dd(A_{\iii|_{n}})_*m(V)-g(V(\iii))\biggr|\leq\varepsilon.
  $$
  Hence, $\lim_{n\to\infty}(A_{\iii|_{n}})_*m$ exists and equals to $\delta_{V(\iii)}$.

  On the other hand, if $\iii\in\Upsilon$, then clearly $\lim_{n\to\infty}(A_{\iii|_{n}})_*m=(A_{\iii|_{k}})_*m$, where $k$ is the smallest $n_0$ satisfying the condition in \eqref{eq:decombsigma2}.
\end{proof}

\begin{proof}[Proof of Theorem~\ref{thm:fortuples}]
  The equivalence of \eqref{thm:ftvvv} and \eqref{thm:fti} follows directly from Corollary~\ref{thm:justdomin2}. By Lemma~\ref{lem:irred}, the equilibrium state $\mu$ is unique and a Gibbs-type measure for $\Phi^s$. Thus, also \eqref{thm:ftv} and \eqref{thm:ftvvv} can be immediately seen to be equivalent.

  Let us show that \eqref{thm:ftvii} implies \eqref{thm:ftv}. Plugging \eqref{thm:ftvii} into \eqref{eq:Gibbstype}, we see that
  $$
    C^{-1}\exp\biggl(-nP(\Phi) + s\sum_{k=0}^{n-1} f(\sigma^k\iii)\biggr) \le \mu([\iii|_n]) \le C\exp\biggl(-nP(\Phi) + s\sum_{k=0}^{n-1} f(\sigma^k\iii)\biggr)
  $$
  holds for every $\iii\in\Sigma$, from which the quasi-Bernoulli property clearly follows.

  It remains to show that \eqref{thm:fti} implies \eqref{thm:ftvii}. By Lemma~\ref{lem:convmeas}, $\nu_{\iii}=\lim_{n\to\infty}(A_{\iii|_{n}})_*m$ exists for every $\iii\in\Sigma$. Define $f\colon\Sigma\to\R$ by setting
  $$
    f(\iii)=\int\log\|A_{\iii|_1}|V\|\dd\nu_{\sigma\iii}(V)
  $$
  for all $\iii \in \Sigma$. Clearly,
  $$
    \int|f(\iii)|\dd\mu(\iii)\leq\iint|\log\|A_{i_0}|V\||\dd\nu_{\sigma\iii}(V)\dd\mu(\iii)\leq\iint C\dd\nu_{\sigma\iii}(V)\dd\mu(\iii)=C,
  $$
  where $C=\log\max\{\max_{i}\{\|A_{i}\|\},\max_i\{\|A_i^{-1}\|\}\}$. Let $\hat\Sigma$ and $\Upsilon$ be as in \eqref{eq:decombsigma1} and \eqref{eq:decombsigma2}, respectively. Since $\mu$ is fully supported, $\mu(\Upsilon)=0$ and $\hat\Sigma$ has full $\mu$ measure. Furthermore, every $\iii \in \hat\Sigma$ satisfies
  $$
    \diam(A_{\iii|_{n}}(\CC))\leq C\tau^{\sharp_{1}\iii|_n+\cdots+\sharp_{M}\iii|_n}\to0
  $$
  as $n\to\infty$.
  Therefore, for $\mu$-almost every $\iii$ and for any sequence $(\jjj_n)_{n\in\N}$ converging to $\iii$ and sufficiently large $n$,
  \begin{align*}
    |f(\iii)-f(\jjj_n)|&=\biggl|\int\log\|A_{\iii|_1}|V\|\dd\nu_{\sigma\iii}(V)-\int\log\|A_{\jjj_n|_1}|V\|\dd\nu_{\sigma\jjj_n}(V)\biggr|\\
    &=\biggl|\log\|A_{\iii|_1}|V(\sigma\iii)\|-\int\log\|A_{\iii|_1}|V\|\dd\nu_{\sigma\jjj_n}(V)\biggr|\\
    &\leq C\dist(\delta_{V(\sigma\iii)},\nu_{\sigma\jjj_n})\leq C'\diam(A_{\sigma\iii\wedge\sigma\jjj_n}(\CC)),
  \end{align*}
  which converges to $0$ as $n\to\infty$. Note that
  $$
    \sum_{k=0}^{n-1}f(\sigma^k\iii)=\int\log\|A_{\iii|_n}|V\|\dd\nu_{\sigma^n\iii}(V)
  $$
  for every $n\in \N$ and $\iii\in\Sigma$.
  By Lemma~\ref{lem:Morris}, there exists $\kappa>0$ such that $\|A_{\iii|_n}\|\geq\|A_{\iii|_n}|V\|\geq\kappa\|A_{\iii|_n}\|$ for all $V\in\CC$. Therefore, \eqref{thm:ftvii} follows.
\end{proof}

The following lemma, which we refer to as the three matrices lemma, is the key observation in the proof of Theorem~\ref{thm:holder}.

\begin{lemma}\label{prop:key}
  If $\A=(A_1,A_2,A_3)\in GL_2(\R)^3$ is such that $A_3=cI$ for some $c \in\R\setminus \{0\}$ and $(A_1,A_2)$ is irreducible and dominated, then for every H\"older continuous potential $f\colon\{1,2,3\}^{\N}\to\R$ and every $C>0$ there exists $\iii\in\{1,2,3\}^{\N}$ and $n\in\N$ such that
  $$
    \Biggl|\sum_{k=0}^{n-1}f(\sigma^k\iii)-\log\|A_{\iii|_n}\|\Biggr|> C.
  $$
\end{lemma}

The proof of the lemma takes several pages. Trying not to disrupt the flow of the proofs in this section, we have postponed it into \S\ref{sec:3m}.

\begin{proof}[Proof of Theorem~\ref{thm:holder}]
  By Lemma~\ref{lem:irred}, the equilibrium state $\mu$ is unique and a Gibbs-type measure for $\Phi^s$. Taking the potential $f$ in \eqref{it:holder3}, it is clear that $\mu$ is Gibbs for the potential $sf$. On the other hand, if $\mu$ is Gibbs for the potential $g$ then $\frac{1}{s}g$ clearly satisfies \eqref{it:holder3}.

  Let us show that \eqref{it:holder2} implies \eqref{it:holder3}. If $\A$ has a strongly invariant multicone $\CC$, then, by e.g.\ \cite[Lemma 2.4]{BaranyRams2017}, there exist H\"older-continuous functions $V\colon\Sigma\to\RP$ and $f\colon\Sigma\to\R$ such that
  \begin{equation}\label{eq:hpot}
    V(\iii)=\bigcap_{n=1}^{\infty}A_{\iii|_n}(\CC) \quad \text{and} \quad f(\iii)=\log\|A_{\iii|_1}|V(\sigma\iii)\|
  \end{equation}
  for all $\iii\in\Sigma$.
  Moreover, by Lemma~\ref{lem:Morris}, there is a constant $C>0$ such that
  $$
    \Biggl|\sum_{k=0}^{m-1}f(\sigma^k\iii)-\log\|A_{\iii|_m}\|\Biggr|\leq C.
  $$
  for all $\iii\in\Sigma$ and $m\in\N$.
  On the other hand, if $\A$ is strongly conformal then, by choosing $f(\iii)=\frac{1}{2}\log|\det(A_{\iii|_1})|$, the claimed properties follow.

  It remains to show that \eqref{it:holder3} implies \eqref{it:holder2}. Let us assume contrarily that there exist a constant $C>0$ and a H\"older-continuous function $f$ such that
  \begin{equation}\label{eq:tocontra2}
    \Biggl|\sum_{k=0}^{n-1}f(\sigma^k\iii)-\log\|A_{\iii|_n}\|\Biggr|\leq C,
  \end{equation}
  for all $\iii\in\Sigma$, $\A$ does not have strongly invariant multicone,
  and $\A$ is not strongly conformal. Thus, by Theorem~\ref{thm:fortuples}, $\A$ can be decomposed into $\A_h\neq\emptyset$ and strongly conformal set $\A_e\neq\emptyset$ such that $\A_h$ has strongly invariant multicone $\CC$ and $A\CC=\CC$ for every $A\in\A_e$. As usual, let $\A_h=\{A_1,\ldots,A_M\}$ and $\A_e=\{A_{M+1},\ldots,A_N\}$. The equilibrium state $\mu$ is a quasi-Bernoulli measure. Recall that, by Proposition~\ref{prop:connect}, $\{|\det(A)|^{-1/2}A:A\in\SS(\A_e)\}$ is finite. Hence, there exists $A_\jjj\in\SS(\A_e)$ such that $A_\jjj=cI$.

Since $\A_h$ is non-empty and $\A$ is irreducible, $X_u(\A)$ and $X_s(\A)$ contain at least two points each. Then there exist four proximal matrices $A_{\iii_1},A_{\iii_2},A_{\iii_3},A_{\iii_4}\in\SS(\A)$ such that $u(\A_{\iii_1})\neq u(\A_{\iii_2})$ and $s(A_{\iii_3})\neq s(A_{\iii_4})$. Taking $q>0$ sufficiently large we have that $A_{\iii_1}^q\CC\cap A_{\iii_2}^q\CC=\emptyset$ and $A_{\iii_3}^{-q}(\CC^o)^c\cap A_{\iii_4}^{-q}(\CC^o)^c=\emptyset$. Clearly, $u(A_{\iii_1}^qA_{\iii_3}^{q})\in A_{\iii_1}^q\CC$ and $u(A_{\iii_2}^qA_{\iii_4}^{q})\in A_{\iii_2}^q\CC$ and so $u(A_{\iii_1}^qA_{\iii_3}^{q})\neq u(A_{\iii_2}^qA_{\iii_4}^{q})$. Similarly, $s(A_{\iii_1}^qA_{\iii_3}^{q})\in A_{\iii_3}^{-q}(\CC^o)^c$ and $s(A_{\iii_2}^qA_{\iii_4}^{q})\in A_{\iii_4}^{-q}(\CC^o)^c$ and so $s(A_{\iii_1}^qA_{\iii_3}^{q})\neq s(A_{\iii_2}^qA_{\iii_4}^{q})$. Thus,
$(A_{\iii_1}^qA_{\iii_3}^{q},A_{\iii_2}^qA_{\iii_4}^{q})$ is dominated and irreducible.

There exist $n_1,n_2,n_3\geq1$ such that $\ell:=n_3|\jjj|=n_1q(|\iii_1|+|\iii_3|)=n_2q(|\iii_2|+|\iii_4|)$. Let us define $\Gamma = \{(\iii_1^q\iii_3^q)^{n_1},(\iii_2^q\iii_4^q)^{n_1},\jjj^{n_3}\}^\N$. By \eqref{eq:tocontra2}, the H\"older continuous potential $h=\sum_{j=0}^{\ell-1}f\circ\sigma^j$ satisfies
  $$
    \Biggl|\sum_{k=0}^{m-1}h(\sigma^k\iii)-\log\|A_{\iii|_m}\|\Biggr|\leq C
  $$
  for all $m\in\N$ and $\iii\in\Gamma$, where $\sigma$ denotes the left-shift operator on $\Gamma$. Since this contradicts Lemma~\ref{prop:key}, we have finished the proof.
\end{proof}

\section{The three matrices lemma}\label{sec:3m}

In this section, we prove Lemma~\ref{prop:key}. Throughout the section, we assume that $\A=(A_1,A_2,A_3)\in GL_2(\R)^3$ is such that $A_3=cI$ for some $c \in \R\setminus\{0\}$, and $(A_1,A_2)$ is irreducible and has a strongly invariant multicone $\CC$. Note that there exists a multicone $\CC_0 \subset \CC^o$ such that $A_i\CC \subset \CC_0$ for $i=1,2$. Without loss of generality, by multiplying the matrix triple $\A$ by $c^{-1}$, we may assume that $c=1$. This does not affect on the existence of a H\"older continuous potential.

For simplicity, let us denote $\Sigma=\{1,2,3\}^{\N}$ and $\Gamma=\{1,2\}^{\N}$. Let the Borel $\sigma$-algebras of $\Sigma$ and $\Gamma$ be $\BB_{\Sigma}$ and $\BB_{\Gamma}$, respectively. As in \eqref{eq:decombsigma2}, let $\Upsilon=\bigcup_{n=0}^{\infty}\bigcup_{\iii\in\Sigma_n}\{\iii3^{\infty}\} \subset \Sigma$ be the countable set of infinite words whose tail consists only $3$'s, and define $\hat\Sigma = \Sigma \setminus \Upsilon$.
Notice that each $\iii\in\hat\Sigma$ can be written in the form $\iii=3^{k_1}i_13^{k_2}i_2\cdots$, where $k_i\in\N\cup\{0\}$ and $i_k\in\{1,2\}$ for all $k \in \N$. Relying on this representation, let us define a function $\kappa\colon \hat\Sigma\to\Gamma$ by setting
$$
  \kappa(3^{k_1}i_13^{k_2}i_2\cdots)=i_1i_2\cdots
$$
for all $\iii\in \hat\Sigma$. The definition of $\kappa$ can be naturally extended to $\Sigma_*$ by $\kappa(3^{k_1}i_13^{k_2},\ldots,i_n3^{k_{n+1}})=(i_1,\ldots,i_n)$ and $\kappa(3^k)=\emptyset$, where $k_i\in\N\cup\{0\}$.

Observe that $\kappa^{-1}(C)$ is a countable union of cylinder sets in $\Sigma$ for every cylinder set $C$ in $\Gamma$. Thus $\kappa\colon(\Sigma,\BB_{\Sigma})\to(\Gamma,\BB_{\Gamma})$ is measurable. With a slight abuse of notation, we denote both left-shift operators on $\Sigma$ and $\Gamma$ by $\sigma$. Finally, let us observe that
\begin{equation}\label{eq:invkappa}
\kappa(\sigma\iii)=
\begin{cases}
  \kappa(\iii), & \mbox{if } \iii|_1=3 \\
  \sigma\kappa(\iii), & \mbox{if } \iii|_1\neq3.
\end{cases}
\end{equation}

Let $\mu_h$ be the unique ergodic Gibbs measure on $\Gamma$ for the H\"older continuous potential $h\colon\Gamma\to\R$ defined by
\begin{equation*}
h(\iii)=\log\|A_{\iii|_1}|V(\sigma\iii)\|,
\end{equation*}
where $V(\iii)=\bigcap_{n=1}^{\infty}A_{\iii|_n}(\CC)$.
Since
\begin{equation*}
  \sum_{k=0}^{m-1} h(\sigma^k\iii) = \log\|A_{\iii|_m}|V(\sigma^m\iii)\|,
\end{equation*}
Lemma~\ref{lem:Morris} implies
$$
  \Biggl|\sum_{k=0}^{m-1}h(\sigma^k\iii)-\log\|A_{\iii|_m}\|\Biggr|\leq C.
$$
for all $\iii\in\Gamma$ and $m\in \N$.

Let us assume contrarily that the statement of Lemma~\ref{prop:key} fails. This means that there is a H\"older continuous potential $f\colon\Sigma\to\R$ and a constant $C>0$ such that
\begin{equation}\label{eq:ass}
  \Biggl|\sum_{k=0}^{n-1}f(\sigma^k\iii)-\log\|A_{\iii|_n}\|\Biggr|\leq C
\end{equation}
for all $n\in\N$ and $\iii\in\Sigma$. Our goal is to show that in this case the Gibbs measure $\mu_h$ is a Bernoulli measure. By Proposition~\ref{prop:bernoulli}, as the tuple $(A_1,A_2)$ is irreducible and contains only proximal matrices, this is a contradiction. We will show this after some auxiliary lemmas.

The proof of the following lemma follows easily from the definition of $\kappa$ and the domination of the tuple $(A_1,A_2)$, and we leave it to the reader.

\begin{lemma}\label{lem:boundonnorm}
  There exists $C>0$ such that
  $$
    \Biggl|\log\|A_{\iii|_n}\|-\sum_{k=0}^{n-1-\sharp_3\iii|_n}h(\sigma^k\kappa(\iii))\Biggr|\leq C.
  $$
  for all $\iii\in\hat\Sigma$ and $n\in\N$.
\end{lemma}

Let $f$ be the H\"older continuous potential in \eqref{eq:ass} and let $\mu_f$ be the unique ergodic Gibbs measure for the potential $f$ on $\Sigma$.
By the definition of the pressure and \eqref{eq:ass}, we have
$$
  P\biggl(\biggl(\sum_{k=0}^{n-1}f\circ\sigma^{k}\biggr)_n\biggr)=P((\iii\mapsto\log\|A_{\iii|_n}\|)_n)=\lim_{n\to\infty}\tfrac{1}{n}\log\sum_{\iii\in\Sigma_n}\|A_{\iii}\|.
$$
Let us denote the common quantity by $Q$. Then by the definition of Gibbs measures \eqref{eq:gibbsmeas}, there exists a constant $C>0$ such that
$$
  C^{-1}\exp\biggl(\sum_{k=0}^{n-1}f(\sigma^k\iii)-nQ\biggr)\leq\mu_f([\iii|_n])\leq C\exp\biggl(\sum_{k=0}^{n-1}f(\sigma^k\iii)-nQ\biggr),
$$
for every $\iii\in\Sigma$. Let us write
$$
  R=\lim_{n\to\infty}\tfrac{1}{n}\log\sum_{\iii\in\Gamma_n}\|A_{\iii}\|.
$$
By a simple calculation, recalling that $A_3=I$, we see that
$$
  Q=\lim_{n\to\infty}\tfrac{1}{n}\log\sum_{\iii\in\Sigma_n}\|A_{\iii}\|=\lim_{n\to\infty}\tfrac{1}{n}\log\sum_{\ell=0}^n\binom{n}{\ell}\sum_{\iii\in\Gamma_{\ell}}\|A_{\iii}\|.
$$
Since for every $\varepsilon>0$ there exists a constant $K>0$ such that
$$
  K^{-1}e^{(R-\varepsilon)\ell}\leq\sum_{\iii\in\Gamma_{\ell}}\|A_{\iii}\|\leq Ke^{(R+\varepsilon)\ell}
$$
for every $\ell\in\N$, we see that
$\log(1+e^{R-\varepsilon})\leq Q\leq \log(1+e^{R+\varepsilon})$. Since $\varepsilon>0$ was arbitrary, we get
\begin{equation}\label{eq:presseq}
Q=\log(1+e^{R}).
\end{equation}

Let us define the Perron-Frobenius operators $\LL_f$ and $\LL_h$ on $\Sigma$ and on $\Gamma$, respectively, for the H\"older-continuous potentials $f$ and $h$ as
$$
  (\LL_f(\psi))(\iii)=\sum_{i=1}^3e^{f(i\iii)}\psi(i\iii) \quad \text{and} \quad (\LL_h(\phi))(\iii)=\sum_{i=1}^2e^{h(i\iii)}\phi(i\iii).
$$
By \cite[Theorem~1.7 and the proof of Theorem~1.16]{Bowen}, there exist unique functions $\psi_f\colon\Sigma\to\R$ and $\phi_h\colon\Gamma\to\R$ (i.e.\ eigenfunctions) and unique probability measures $\nu_f$ on $\Sigma$ and $\nu_h$ on $\Gamma$ (i.e.\ eigenmeasures) such that
$$
  \LL_f(\psi_f)=e^{Q}\psi_f,\quad \LL_h(\phi_h)=e^{R}\phi_h,\quad \LL_f^*\nu_f=e^{Q}\nu_f,\quad\LL_h^*(\nu_h)=e^{R}\nu_h,
$$
and $\int \psi_f\dd\nu_f=1=\int\phi_h\dd\nu_h$.
Moreover, by \cite[Lemmas~1.8 and 1.10]{Bowen}, the potentials $\log\psi_f$ and $\log\phi_h$ are H\"older continuous. By induction, it is easy to see that for any function $\varphi$
\[
\begin{split}
(\LL_f^n(\varphi))(\iii)&=\sum_{j_1}e^{f(j_1\iii)}(\LL_f^{n-1}(\varphi))(j_1\iii)=\sum_{j_1,j_2}e^{f(j_1\iii)}e^{f(j_2j_1\iii)}(\LL_f^{n-2}(\varphi))(j_2j_1\iii)\\
&=\sum_{j_1,\ldots,j_n}e^{f(j_1\iii)+\cdots+f(j_n\ldots j_1\iii)}\varphi(j_n\ldots j_1\iii)=\sum_{\kkk\in\Sigma_n}e^{\sum_{k=0}^{n-1}f(\sigma^k\kkk\iii)}\varphi(\kkk\iii).
\end{split}
\]

By \cite[Proposition~1.14]{Bowen} and the uniqueness of the ergodic Gibbs measure,
\begin{equation}\label{eq:conf}
  \mu_f(B)=\int_B\psi_f\dd\nu_f \quad \text{and} \quad \mu_h(B')=\int_{B'}\phi_h\dd\nu_h
\end{equation}
for every $B\in\BB_{\Sigma}$ and $B'\in\BB_{\Gamma}$. Thus, for any $\jjj\in\Sigma_n$ and every $B\in\BB_{\Sigma}$
\begin{equation} \label{eq:thisone}
\begin{split}
  \mu_f([\jjj] \cap \sigma^{-|\jjj|}(B)) &= \int\psi_f(\iii)\mathds{1}_{[\jjj] \cap \sigma^{-|\jjj|}(B)}(\iii)\dd\nu_f(\iii)\\
  &= \int\psi_f(\iii)\mathds{1}_{[\jjj] \cap \sigma^{-|\jjj|}(B)}(\iii)e^{-nQ}\dd(\LL_f^*)^n(\nu_f)(\iii)\\
  &= \int\LL_f^n(\psi_f \mathds{1}_{[\jjj] \cap \sigma^{-|\jjj|}(B)})(\iii)e^{-nQ}\dd\nu_f(\iii)\\
  &=\int\sum_{\kkk\in\Sigma_n}\exp\biggl(\sum_{k=0}^{n-1}f(\sigma^k\kkk\iii)-nQ\biggr)\psi_f(\kkk\iii)\mathds{1}_{[\jjj] \cap \sigma^{-|\jjj|}(B)}(\kkk\iii)\dd\nu_f(\iii)\\
  &=\int_B\exp\biggl(\sum_{k=0}^{n-1}f(\sigma^k\jjj\iii)-nQ\biggr)\psi_f(\jjj\iii)\dd\nu_f(\iii)\\
  &=\int_B\exp\biggl(\sum_{k=0}^{n-1}f(\sigma^k\jjj\iii)-nQ\biggr)\frac{\psi_f(\jjj\iii)}{\psi_f(\iii)}\dd\mu_f(\iii).
\end{split}
\end{equation}
Let $\hat f(\iii)=f(\iii)+\log\psi_f(\iii)-\log\psi_f(\sigma\iii)$. Since $\log\psi_f$ is H\"older continuous and thus, uniformly bounded over $\Sigma$, there exists $C>0$ such that
$$
  \Biggl|\sum_{k=0}^{n-1}\hat f(\sigma^k\iii)-\log\|A_{\iii|_n}\|\Biggr|\leq C
$$
for all $n\in\N$ and $\iii\in\Sigma_n$.
By \eqref{eq:thisone},
\begin{equation}\label{eq:thison2} 
  \mu_{f}([\iii] \cap \sigma^{-|\iii|}(B))=\int_{B}\exp\left(\sum_{k=0}^{|\iii|-1}\hat f(\sigma^k\jjj\iii)-|\iii|Q\right)\dd\mu_{f}(\iii).
\end{equation}
for all $\iii\in\Sigma_*$ and $B\in\BB_{\Sigma}$.

Let us denote the ratio $(1+e^R)^{-1}$ by $q$. Define
$$
  \eta([\iii])=q^{\sharp_3\iii}(1-q)^{|\iii|-\sharp_3\iii}\mu_h([\kappa(\iii)])
$$
for all $\iii\in\Sigma_*$ and notice that
\begin{equation}\label{eq:toeta}
\begin{split}
  \sum_{i=1}^3\eta([\iii i])&=\sum_{i=1}^2q^{\sharp_3\iii}(1-q)^{|\iii|+1-\sharp_3\iii}\mu_h([\kappa(\iii i)])+q^{\sharp_3\iii+1}(1-q)^{|\iii|-\sharp_3\iii}\mu_h([\kappa(\iii)])\\
  &=q^{\sharp_3\iii}(1-q)^{|\iii|+1-\sharp_3\iii}\sum_{i=1}^2\mu_h([\kappa(\iii)i])+q^{\sharp_3\iii+1}(1-q)^{|\iii|-\sharp_3\iii}\mu_h([\kappa(\iii)])\\
  &=q^{\sharp_3\iii}(1-q)^{|\iii|-\sharp_3\iii}\mu_h([\kappa(\iii)])(1-q+q)=\eta([\iii]).
\end{split}
\end{equation}
for all $\iii\in\Sigma_*$.
Thus, by Kolmogorov's extension theorem, $\eta$ can be extended to a probability measure on $(\Sigma,\BB_{\Sigma})$. We shall denote the extension by $\eta$ too. The following lemma shows that $\eta$ is ergodic.

\begin{lemma}\label{lem:propeta}
  The measure $\eta$ is $\sigma$-invariant and mixing on $\Sigma$.
\end{lemma}

\begin{proof}
  Since $\mu_h$ is $\sigma$-invariant, the proof of $\sigma$-invariance of $\eta$ is similar to \eqref{eq:toeta}, and therefore, we omit it. To prove that $\eta$ is mixing, it is sufficient to show that
  $$
    \lim_{n\to\infty}\eta([\iii]\cap\sigma^{-n}[\jjj])=\eta([\iii])\eta([\jjj]).
  $$
  for all $\iii,\jjj\in\Sigma_*$.
  Let $n>|\iii|$ and observe that
  \begin{align*}
    \eta([\iii]\cap\sigma^{-n}[\jjj])&=\sum_{\hhh\in\Sigma_{n-|\iii|}}\eta([\iii\hhh\jjj])=\sum_{\hhh\in\Sigma_{n-|\iii|}}q^{\sharp_3\iii+\sharp_3\jjj+\sharp_3\hhh}(1-q)^{|\iii|-\sharp_3\iii+|\jjj|-\sharp_3\jjj+|\hhh|-\sharp_3\hhh}\mu_h([\kappa(\iii\hhh\jjj)])\\
    &=q^{\sharp_3\iii+\sharp_3\jjj}(1-q)^{|\iii|-\sharp_3\iii+|\jjj|-\sharp_3\jjj}\sum_{\hhh\in\Sigma_{n-|\iii|}}q^{\sharp_3\hhh}(1-q)^{|\hhh|-\sharp_3\hhh}\mu_h([\kappa(\iii)\kappa(\hhh)\kappa(\jjj)])\\
    &=q^{\sharp_3\iii+\sharp_3\jjj}(1-q)^{|\iii|-\sharp_3\iii+|\jjj|-\sharp_3\jjj}\sum_{\ell=0}^{n-|\iii|}\binom{n-|\iii|}{\ell}q^{n-|\iii|-\ell}(1-q)^{\ell}\sum_{\kkk\in\Gamma_{\ell}}\mu_h([\kappa(\iii)\kkk\kappa(\jjj)]).
  \end{align*}
  Hence,
  \begin{equation}\label{eq:refasked}
  \frac{\eta([\iii]\cap\sigma^{-n}[\jjj])}{\eta([\iii])\eta([\jjj])}=\sum_{\ell=0}^{n-|\iii|}\binom{n-|\iii|}{\ell}q^{n-|\iii|-\ell}(1-q)^{\ell}\frac{\mu_h([\kappa(\iii)\cap\sigma^{-\ell-|\kappa(\iii)|}\kappa(\jjj)])}{\mu_h([\kappa(\iii)])\mu_h([\kappa(\jjj)])}.
  \end{equation}

  By \cite[Proposition~1.14]{Bowen}, the measure $\mu_h$ is mixing. Thus, for every $\varepsilon>0$ there exists $N$ such that if $\ell\geq N$, then
  $$
    e^{-\varepsilon}\leq\frac{\mu_h([\kappa(\iii)]\cap\sigma^{-\ell-|\kappa(\iii)|}[\kappa(\jjj)])}{\mu_h([\kappa(\iii)])\mu_h([\kappa(\jjj)])}\leq e^{\varepsilon}.
  $$
  Hence, for $n>N+|\iii|$ we get
  \begin{align*}
    \sum_{\ell=0}^{n-|\iii|}&\binom{n-|\iii|}{\ell}q^{n-|\iii|-\ell}(1-q)^{\ell}\frac{\mu_h([\kappa(\iii)\cap\sigma^{-\ell-|\kappa(\iii)|}\kappa(\jjj)])}{\mu_h([\kappa(\iii)])\mu_h([\kappa(\jjj)])}\\	&\leq e^{\varepsilon}\sum_{\ell=N}^{n-|\iii|}\binom{n-|\iii|}{\ell}q^{n-|\iii|-\ell}(1-q)^{\ell}+\mu_h([\kappa(\jjj)])^{-1}\sum_{\ell=0}^{N-1}\binom{n-|\iii|}{\ell}q^{n-|\iii|-\ell}(1-q)^{\ell}\\
    &\leq e^{\varepsilon}+\mu_h([\kappa(\jjj)])^{-1}N(n-|\iii|)^N(1-q)^{n-|\iii|-N},
  \end{align*}
  where in the last inequality we used $\binom{n}{k} \leq n^k$. By a similar argument,
  \begin{align*}
    \sum_{\ell=0}^{n-|\iii|}\binom{n-|\iii|}{\ell}&q^{n-|\iii|-\ell}(1-q)^{\ell}\frac{\mu_h([\kappa(\iii)\cap\sigma^{-\ell-|\kappa(\iii)|}\kappa(\jjj)])}{\mu_h([\kappa(\iii)])\mu_h([\kappa(\jjj)])}\\
    &\geq e^{-\varepsilon}\sum_{\ell=N}^{n-|\iii|}\binom{n-|\iii|}{\ell}q^{n-|\iii|-\ell}(1-q)^{\ell}\geq e^{-\varepsilon}-N(n-|\iii|)^N(1-q)^{n-|\iii|-N}.
  \end{align*}
  By \eqref{eq:refasked} and letting $n\to\infty$, we see that for every $\varepsilon>0$
  \begin{equation*}
    e^{-\varepsilon}\leq\lim_{n\to\infty}\frac{\eta([\iii]\cap\sigma^{-n}[\jjj])}{\eta([\iii])\eta([\jjj])}\leq e^{\varepsilon}.
  \end{equation*}
  Since $\varepsilon>0$ was arbitrary, the definition of $\eta$ finishes the proof.
\end{proof}

\begin{proposition}\label{prop:etaismuf}
  If $\eta$ and $\mu_f$ are as above, then $\eta=\mu_f$.
\end{proposition}

\begin{proof}
  Since both $\eta$ and $\mu_f$ are ergodic measures, it suffices to show that they are equivalent. By Lemma~\ref{lem:boundonnorm} and our assumption \eqref{eq:ass} on $f$,
  \begin{align*}
    \eta([\iii])&=q^{\sharp_3\iii}(1-q)^{|\iii|-\sharp_3\iii}\mu_h([\kappa(\iii)])\\
    &\leq Cq^{\sharp_3\iii}(1-q)^{|\iii|-\sharp_3\iii}\exp\biggl(\sum_{k=0}^{|\kappa(\iii)|-1}h(\sigma^k\kappa(\iii)\jjj)-|\kappa(\iii)|R\biggr)\\
    &\leq C'(1-q)^{|\iii|-\sharp_3\iii}\|A_{\iii}\|\exp\bigl(-(|\iii|-\sharp_3\iii)R-\sharp_3\iii\log(1+e^R)\bigr)\\
    &=C'\|A_{\iii}\|\exp(-|\iii|\log(1+e^R))\\
    &\leq C''\exp\biggl(\sum_{k=0}^{|\iii|-1}f(\sigma^k\iii\jjj')-|\iii|\log(1+e^R)\biggr)\leq C'''\mu_f([\iii])
  \end{align*}
  for all $\iii\in\Sigma_*$.
  The other inequality follows by a similar argument. Since for every cylinder set $[\iii]$, the ratio $\eta([\iii])/\mu_f([\iii])$ is bounded away from $0$ and $\infty$ uniformly, the statement follows.
\end{proof}

By \eqref{eq:thison2} and \eqref{eq:presseq},
\begin{equation*} 
  \mu_{f}([i] \cap \sigma^{-1}(B))=\int_{B}\exp(\hat f(i\iii)-Q)\dd\mu_{f}(\iii)=\frac{1}{1+e^R}\int_{B}\exp(\hat f(i\iii))\dd\mu_{f}(\iii)
\end{equation*}
for all $i\in\{1,2,3\}$ and $B\in\BB_{\Sigma}$.
By Proposition~\ref{prop:etaismuf} and recalling the definition of $\eta$, we have
\begin{equation*} 
  \mu_{f}([3] \cap \sigma^{-1}(B))=\eta([3] \cap \sigma^{-1}(B))=\frac{1}{1+e^{R}}\eta(B)=\frac{1}{1+e^{R}}\mu_{f}(B).
\end{equation*}
Since $\hat f$ is H\"older continuous and the above two equations hold for every  $B\in\BB_{\Sigma}$, we conclude that
\begin{equation}\label{eq:f'piecewise}
  \hat f(\iii)=0
\end{equation}
for all $\iii\in[3]$. By \eqref{eq:thison2}, we have
\begin{align*}
  \mu_f([i_13^{k_1}\cdots i_n3^{k_n}])&=\int \exp\biggl(\sum_{\ell=0}^{k_1+\cdots+k_n+n-1}\hat f(\sigma^{\ell}i_13^{k_1}\cdots i_n3^{k_n}\jjj)-(k_1+\cdots+k_n+n)Q\biggr)\dd\mu_f(\jjj)\\
  &=q^{k_1+\cdots+k_n}\int \exp(\hat f(i_13^{k_1}\cdots\jjj)+\cdots+\hat f(i_n3^{k_n}\jjj)-nQ)\dd\mu_f(\iii)
\end{align*}
for every $k_1,\ldots,k_n\in\N$, $i_1,\dots,i_n\in\{1,2\}$, and $n\in\N$.
Since $\hat f$ is H\"older continuous, we have
$$
  \lim_{k_{1},\ldots,k_n\to\infty}\hat f(i_{\ell}3^{k_{\ell}}\cdots i_n3^{k_n}\jjj)=\hat f(i_{\ell}3^{\infty})
$$
uniformly for all $\ell\in\{1,\ldots,n\}$ and $\jjj\in\Sigma$. Hence by the dominated convergence theorem
$$
  \lim_{k_{1},\ldots,k_n\to\infty}\frac{\mu_f([i_13^{k_1}\cdots i_n3^{k_n}])}{q^{k_1+\cdots+k_n}}=\prod_{\ell=1}^ne^{\hat f(i_{\ell}3^{\infty})-Q}.
$$
On the other hand, by the definition of $\eta$ and Proposition~\ref{prop:etaismuf},
$$
  \frac{\mu_f([i_13^{k_1}\cdots i_n3^{k_n}])}{q^{k_1+\cdots+k_n}}=(1-q)^{n}\mu_h([i_1\cdots i_n]).
$$
It follows that
$$
  \mu_h([i_1\cdots i_n])=\prod_{\ell=1}^ne^{\hat f(i_{\ell}3^{\infty})-R}
$$
and hence, $\mu_h$ is a Bernoulli measure. This contradicts Proposition~\ref{prop:bernoulli} and finishes the proof of Lemma~\ref{prop:key}.

\begin{acknowledgements}
  Bal\'azs B\'ar\'any acknowledges support from the grants NKFI PD123970, OTKA K123782, and the J\'anos Bolyai Research Scholarship of the Hungarian Academy of Sciences. Antti K\"aenm\"aki was supported by the Finnish Center of Excellence in Analysis and Dynamics Research. Ian Morris was supported by the Leverhulme Trust (Research Project Grant number RPG-2016-194). All the authors were partially supported by the ERC grant 306494. The research was started in the Hebrew University of Jerusalem. The authors thank the HUJI, and especially Professor Michael Hochman, for warm hospitality. B\'ar\'any and K\"aenm\"aki also thank the Institut Mittag-Leffler, where the paper was finished. Finally, the authors thank the anonymous referee for the careful reading and valuable comments which improved the paper. 
\end{acknowledgements}


\begin{thebibliography}{10}

\bibitem{AvilaBochiYoccoz2010}
A.~Avila, J.~Bochi, and J.-C. Yoccoz.
\newblock Uniformly hyperbolic finite-valued {${\rm
  SL}(2,\mathbb{R})$}-cocycles.
\newblock {\em Comment. Math. Helv.}, 85(4):813--884, 2010.

\bibitem{BaranyKaenmakiKoivusalo2017}
B.~B{\'a}r{\'a}ny, A.~K{\"a}enm{\"a}ki, and H.~Koivusalo.
\newblock Dimension of self-affine sets for fixed translation vectors.
\newblock {\em J. Lond. Math. Soc. (2)}, 98(1):223--252, 2018.

\bibitem{BaranyRams2017}
B.~B{\'a}r{\'a}ny and M.~Rams.
\newblock Dimension maximizing measures for self-affine systems.
\newblock {\em Trans. Amer. Math. Soc.}, 370(1):553--576, 2018.

\bibitem{Barreira06}
L.~Barreira.
\newblock Nonadditive thermodynamic formalism: equilibrium and {G}ibbs
  measures.
\newblock {\em Discrete Contin. Dyn. Syst.}, 16(2):279--305, 2006.

\bibitem{Barreira2010}
L.~Barreira.
\newblock Almost additive thermodynamic formalism: some recent developments.
\newblock {\em Rev. Math. Phys.}, 22(10):1147--1179, 2010.

\bibitem{BarreiraDoutor09}
L.~Barreira and P.~Doutor.
\newblock Almost additive multifractal analysis.
\newblock {\em J. Math. Pures Appl. (9)}, 92(1):1--17, 2009.

\bibitem{BochiGourmelon09}
J.~Bochi and N.~Gourmelon.
\newblock Some characterizations of domination.
\newblock {\em Math. Z.}, 263(1):221--231, 2009.

\bibitem{BochiMorris15}
J.~Bochi and I.~D. Morris.
\newblock Continuity properties of the lower spectral radius.
\newblock {\em Proc. Lond. Math. Soc. (3)}, 110(2):477--509, 2015.

\bibitem{BochiViana2005}
J.~Bochi and M.~Viana.
\newblock The {L}yapunov exponents of generic volume-preserving and symplectic
  maps.
\newblock {\em Ann. of Math. (2)}, 161(3):1423--1485, 2005.

\bibitem{BomfimVarandas15}
T.~Bomfim and P.~Varandas.
\newblock Multifractal analysis of the irregular set for almost-additive
  sequences via large deviations.
\newblock {\em Nonlinearity}, 28(10):3563--3585, 2015.

\bibitem{BonattiDiazPujals2003}
C.~Bonatti, L.~J. D{\'i}az, and E.~R. Pujals.
\newblock A {$C^1$}-generic dichotomy for diffeomorphisms: weak forms of
  hyperbolicity or infinitely many sinks or sources.
\newblock {\em Ann. of Math. (2)}, 158(2):355--418, 2003.

\bibitem{Bowen}
R.~Bowen.
\newblock {\em Equilibrium states and the ergodic theory of {A}nosov
  diffeomorphisms}, volume 470 of {\em Lecture Notes in Mathematics}.
\newblock Springer-Verlag, Berlin, revised edition, 2008.
\newblock With a preface by David Ruelle, Edited by Jean-Ren\'e Chazottes.

\bibitem{Cao13}
Y.~Cao.
\newblock Dimension spectrum of asymptotically additive potentials for {$C^1$}
  average conformal repellers.
\newblock {\em Nonlinearity}, 26(9):2441--2468, 2013.

\bibitem{CaoFengHuang}
Y.~Cao, D.-J. Feng, and W. Huang.
\newblock The thermodynamic formalism for sub-additive potentials.
\newblock {\em Discrete Contin. Dyn. Syst.}, 20(3):639--657, 2008.

\bibitem{Feng2009}
D.-J. Feng.
\newblock Lyapunov exponents for products of matrices and multifractal
  analysis. {II}. {G}eneral matrices.
\newblock {\em Israel J. Math.}, 170:355--394, 2009.

\bibitem{FengHuang10}
D.-J. Feng and W.~Huang.
\newblock Lyapunov spectrum of asymptotically sub-additive potentials.
\newblock {\em Comm. Math. Phys.}, 297(1):1--43, 2010.

\bibitem{FengKaenmaki2011}
D.-J. Feng and A.~K{\"a}enm{\"a}ki.
\newblock Equilibrium states of the pressure function for products of matrices.
\newblock {\em Discrete Contin. Dyn. Syst.}, 30(3):699--708, 2011.

\bibitem{FraserJordanJurga2017}
J.~Fraser, T.~Jordan, and N.~Jurga.
\newblock Dimensions of equilibrium measures on a class of planar self-affine
  sets.
  \newblock {\em J. Fractal Geom.}
\newblock To appear, available at arXiv:1706.06833, 2017.

\bibitem{FurstenbergKesten1960}
H.~Furstenberg and H.~Kesten.
\newblock Products of random matrices.
\newblock {\em Ann. Math. Statist.}, 31:457--469, 1960.

\bibitem{HueterLalley95}
I.~Hueter and S.~P. Lalley.
\newblock Falconer's formula for the {H}ausdorff dimension of a self-affine set
  in {${\bf R}^2$}.
\newblock {\em Ergodic Theory Dynam. Systems}, 15(1):77--97, 1995.

\bibitem{IommiYayama12}
G.~Iommi and Y.~Yayama.
\newblock Almost-additive thermodynamic formalism for countable {M}arkov
  shifts.
\newblock {\em Nonlinearity}, 25(1):165--191, 2012.

\bibitem{IommiYayama17}
G.~Iommi and Y.~Yayama.
\newblock Weak {G}ibbs measures as {G}ibbs measures for asymptotically additive
  sequences.
\newblock {\em Proc. Amer. Math. Soc.}, 145(4):1599--1614, 2017.

\bibitem{Jungers12}
R.~M. Jungers.
\newblock On asymptotic properties of matrix semigroups with an invariant cone.
\newblock {\em Linear Algebra Appl.}, 437(5):1205--1214, 2012.

\bibitem{Kaenmaki2004}
A.~K{\"a}enm{\"a}ki.
\newblock On natural invariant measures on generalised iterated function
  systems.
\newblock {\em Ann. Acad. Sci. Fenn. Math.}, 29(2):419--458, 2004.

\bibitem{KaenmakiLi2017}
A.~K\"aenm\"aki and B.~Li.
\newblock Genericity of dimension drop on self-affine sets.
\newblock {\em Statist. Probab. Lett.}, 126:18--25, 2017.

\bibitem{KaenmakiReeve2014}
A.~K{\"a}enm{\"a}ki and H.~W.~J. Reeve.
\newblock Multifractal analysis of {B}irkhoff averages for typical infinitely
  generated self-affine sets.
\newblock {\em J. Fractal Geom.}, 1(1):83--152, 2014.

\bibitem{KaenmakiVilppolainen2010}
A.~K{\"a}enm{\"a}ki and M.~Vilppolainen.
\newblock Dimension and measures on sub-self-affine sets.
\newblock {\em Monatsh. Math.}, 161(3):271--293, 2010.

\bibitem{Mane1978}
R.~Ma\~n\'e.
\newblock Contributions to the stability conjecture.
\newblock {\em Topology}, 17(4):383--396, 1978.

\bibitem{Mane1984}
R.~Ma\~n\'e.
\newblock Oseledec's theorem from the generic viewpoint.
\newblock In {\em Proceedings of the {I}nternational {C}ongress of
  {M}athematicians, {V}ol.\ 1, 2 ({W}arsaw, 1983)}, pages 1269--1276. PWN,
  Warsaw, 1984.

\bibitem{Morris2016}
I.~D. Morris.
\newblock Ergodic properties of matrix equilibrium states.
\newblock {\em Ergodic Theory Dynam. Systems}, 38(6):2295--2320, 2018.

\bibitem{Pollicott10}
M.~Pollicott.
\newblock Maximal {L}yapunov exponents for random matrix products.
\newblock {\em Invent. Math.}, 181(1):209--226, 2010.

\bibitem{ProtVoy}
V. Yu. Protasov and A. S. Voynov.,
\newblock Matrix semigroups with constant spectral radius.
\newblock {\em Linear Algebra Appl.}, 513:376--408, 2017.


\bibitem{Yayama16}
Y.~Yayama.
\newblock On factors of {G}ibbs measures for almost additive potentials.
\newblock {\em Ergodic Theory Dynam. Systems}, 36(1):276--309, 2016.

\end{thebibliography}


\end{document}